\documentclass[a4paper, 11pt]{article}
\title{Products of real equivariant weight filtrations}
\author{Fabien Priziac}
\date{}
\usepackage{amsfonts}
\usepackage{amsmath}
\usepackage{amsthm}
\usepackage{amssymb}
\usepackage[all]{xy}
\usepackage{graphicx}
\usepackage{lscape}

\newtheorem{de}{Definition}[section]
\newtheorem{theo}[de]{Theorem}
\newtheorem{prop}[de]{Proposition}

\newtheorem{defprop}[de]{Definition and Proposition}

\newtheorem{lem}[de]{Lemma}

\newcommand{\Hom}{\text{Hom}}

\theoremstyle{remark}
\newtheorem{rem}[de]{Remark}
\newtheorem{ex}[de]{Example}

\begin{document}

\maketitle

\begin{abstract}
We first show the existence of a weight filtration on the equivariant cohomology of real algebraic varieties equipped with the action of a finite group, by applying group cohomology to the dual geometric filtration. We then prove the compatibility of the equivariant weight filtrations and spectral sequences with K\"unneth isomorphism, cup and cap products, from the filtered chain level. We finally induce the usual formulae for the equivariant cup and cap products from their analogs on the non-equivariant weight spectral sequences. 
\end{abstract}

\section{Introduction}

In \cite{Del}, P. Deligne established the existence of a filtration on the rational cohomology with compact supports of complex algebraic varieties, which is trivial on compact nonsingular varieties, additive and compatible with resolutions of singularities : the weight filtration. Using F. Guill\'en and V. Navarro Aznar's work on cubical hyperresolutions (\cite{GNAPP}, \cite{GNA}), B. Totaro introduced in \cite{Tot} an analog of the weight filtration for real algebraic varieties, defined on their cohomology with compact supports and Borel-Moore homology with coefficients in $\mathbb{Z}_2 := \mathbb{Z}/2 \mathbb{Z}$. Unlike the complex case, the spectral sequence associated to the real weight filtration does not degenerate at page two in general. Moreover, it contains important additive invariants of real algebraic varieties : the virtual Betti numbers (\cite{MCP-VB}). In \cite{MCP}, C. McCrory and A. Parusi\'nski proved that the homological real weight spectral sequence and filtration can be induced by a filtered chain complex, the geometric filtration, defined on semialgebraic chains with closed supports. The geometric filtration is itself additive, compatible with resolutions of singularities and furthermore functorial with respect to continuous proper maps with $\mathcal{AS}$-graph (\cite{Kur}, \cite{KP}). In particular, this last fact allows to show that the virtual Betti numbers are invariant under homeomorphisms with $\mathcal{AS}$-graph.

The cohomological counterpart of McCrory and Parusi\'nki's work is tackled in \cite{LP} : the cohomological real weight spectral sequence and filtration can be induced by a dualization of the geometric filtration. This is used to prove in particular that the cohomological and homological real weight spectral sequences and filtrations are dual to one another. The second part of \cite{LP} deals with the issue of the compatibility of the geometric and dual geometric filtrations and the induced weight spectral sequences and filtrations with products. Morphisms defined on the filtered chain level induce K\"unneth isomorphisms as well as cup and cap products on the weight spectral sequences. In particular, obstructions for Poincar\'e duality to be an isomorphism are extracted.

Let us now consider real algebraic varieties equipped with the action of a finite group. We can equip the geometric filtration with the action induced by functoriality (\cite{Pri-CA}). Furthermore, applying to this ``geometric filtration with action'' a functor which computes group homology with values in a chain complex, we can obtain a new chain complex which induces an analog of the weight filtration on the equivariant homology of real algebraic varieties with action defined in \cite{VH}~: the homological equivariant weight filtration (\cite{Pri-EWF}). Significative differences appear between the associated equivariant weight spectral sequence and the non-equivariant one. In particular, it is not left-bounded and, even for compact nonsingular varieties, it does not degenerate at page two in general. Consequently, we can not recover additive invariants directly from the equivariant weight spectral sequence.

In this paper, we first define and study the cohomological equivariant version of the real weight filtration. The dual geometric filtration of a real algebraic variety with action can be equipped with the action induced by functoriality (Definition \ref{defdugeofilact}). We can then apply to this ``dual geometric filtration with action'' a functor which computes group cohomology with values in a cochain complex (Definition and Proposition \ref{defproplcohom}) in order to induce a weight filtration on the equivariant cohomology of real algebraic varieties with action. In addition, similarly to what was done in \cite{MCP}, \cite{LP}, \cite{Pri-CA} and \cite{Pri-EWF}, we show that the dual geometric filtration with action and the induced cohomological equivariant geometric filtration are unique up to filtered quasi-isomorphisms with additivity, acyclicity and triviality properties (Theorem \ref{wcactcohom} and Proposition \ref{dgcwc}, Theorem \ref{unicoequivwc} and Proposition \ref{propcoequivgeowc}), using a version with action (\cite{Pri-CA}) of an extension criterion of Guill\'en and Navarro Aznar (\cite{GNA}). We also remark that, contrary to the non-equivariant ones (\cite{LP}), the cohomological and homological equivariant weight spectral sequences and filtrations are not dual to one another in general (Example \ref{excircle1}).

In a second part, we use the work on products of \cite{LP} to induce a K\"unneth isomorphism (Theorems \ref{thequivhomprod} and \ref{thequivcohomprod}), a cup product (Theorem \ref{theoequivcupprod}) and a cap product (Theorem \ref{theoequivcupprod}) on the equivariant weight spectral sequences, via morphisms defined on the filtered chain level. Precisely, we first check that the filtered chain morphisms used in \cite{LP} to define the products on the weight spectral sequences are equivariant, so that we can apply the functor computing group (co)homology. We then use the facts that, when the considered groups are of finite order, this last functor is itself compatible with products (Propositions \ref{isolprod} and \ref{prodlhom}), and that it is also functorial with respect to group morphisms (Proposition \ref{functgroup}). An important part of the study consists in proving the usual properties of cup and cap products for the equivariant cup and cap products on the equivariant weight spectral sequences (Theorems \ref{comassoccupg}, \ref{comassoc2cupg} and \ref{functcupg}, Theorems \ref{prop1equivcap} and \ref{prop2equivcap}). To this end, we consider further spectral sequences (induced by group (co)homology) which converges to the page two of the equivariant weight spectral sequences and which allow to carry the properties of the non-equivariant weight spectral sequences.

We begin this paper by defining the dual geometric filtration with action of real algebraic varieties equipped with a finite group action, by functoriality. We then prove its uniqueness up to equivariant filtered quasi-isomorphism with triviality, acyclicity and additivity properties.

In section \ref{seccohomequivwf}, we define a functor which computes group cohomology with coefficients in a cochain complex and equivariant cohomology of real algebraic varieties with action. We apply this functor to the dual geometric filtration with action and study the resulting filtered cochain complex, notably the associated spectral sequences and the induced cohomological equivariant weight filtration on the equivariant cohomology of real algebraic varieties with action.

In section \ref{sectprod}, we induce equivariant K\"unneth isomorphisms (subsection \ref{subsectku}), cup products (subsection \ref{subsectcup}) and cap products (subsection \ref{subsectcap}) on the equivariant weight spectral sequences, as well as these products' usual properties, from the products on the weight spectral sequences with action. We also use the compatibility of group cohomology with products as well as the functoriality of group cohomology with respect to homomorphisms of groups.

\section{Cohomological weight complex with action} \label{cohcompl_act}

Let $G$ be a finite group.

As in \cite{Pri-CA}, we use the fonctoriality of the cohomological weight complex of \cite{LP} in order to define a cohomological weight complex with action of $G$ on the category of real algebraic $G$-varieties. Its uniqueness up to filtered quasi-isomorphism will be given by Th\'eor\`eme 3.5 of \cite{Pri-CA}, which is a version with action of the extension criterion of F. Guill\'en and V. Navarro Aznar (\cite{GNA}). 

First, we make precise the framework and notations (inspired by the ones in \cite{GNA}) we are going to work with throughout this paper. In this article, by a real algebraic variety, we mean a reduced separated scheme of finite type over $\mathbb{R}$, and by an action of $G$ on a real algebraic variety $X$, we mean an an action by isomorphisms of schemes such that the orbit of any point in X is contained in an affine open subscheme.

\begin{de} We denote by 
\begin{itemize}
	\item $\mathbf{Sch}_c^G(\mathbb{R})$ the category of real algebraic $G$-varieties -that is, by definition, real algebraic varieties equipped with an action of $G$- and equivariant regular proper morphisms,
	\item $\mathbf{Reg}_{comp}^G(\mathbb{R})$ the full subcategory of compact nonsingular real algebraic $G$-varieties,
	\item $\mathbf{V}^G(\mathbb{R})$ the full subcategory of projective nonsingular real algebraic $G$-varieties.
\end{itemize}

We also denote by
\begin{itemize}
	\item $\mathfrak{C}^G$ the category of filtered bounded cochain $G$-complexes of $\mathbb{Z}_2$-vector spaces -that is, by definition, bounded cochain $G$-complexes of $\mathbb{Z}_2$-vector spaces equipped with a decreasing bounded filtration by cochain $G$-complexes with equivariant inclusions- and equivariant morphisms of filtered cochain complexes, 
	\item $\mathfrak{D}^G$ the category of bounded cochain $G$-complexes of $\mathbb{Z}_2$-vector spaces and equivariant morphisms of cochain complexes.
\end{itemize}
\end{de}

To any real algebraic variety $X$, we can associate its semialgebraic cochain complex with closed supports and $\mathbb{Z}_2$-coefficients $C^*(X)$ (see \cite{LP}) : it is by definition the dual cochain complex of the semialgebraic chain complex with closed supports $C_*(X)$ of $X$ (we refer to \cite{MCP} for the precise definition of $C_*(X)$). This cochain complex $C^*(X)$ computes the cohomology with compact supports of $X(\mathbb{R})$ with coefficients in $\mathbb{Z}_2$, which we denote simply by $H^*(X)$ in this paper. If $X$ is now a real algebraic $G$-variety, the action of $G$ on $X$ induces, thanks to the contravariant functoriality of the semialgebraic cochain complex $C^*$, an action on the cochain complex $C^*(X)$, which becomes a cochain $G$-complex, that is an object of $\mathfrak{D}^G$.
\\

Furthermore, the semialgebraic cochain complex $C^*(X)$ of $X$ can be equipped with the dual geometric filtration $\mathcal{G}^{\bullet}$, defined in \cite{LP} in the following way :
$$\mathcal{G}^p C^q(X) := \{ \varphi \in C^q(X)~|~ \varphi \equiv 0 \mbox{ on } \mathcal{G}_{p-1} C_q(X) \}$$ 
where $\mathcal{G}_{\bullet} C_*(X)$ is the geometric filtration on the semialgebraic chains with closed supports of $X$ defined in \cite{MCP}. The dual geometric filtration is a decreasing filtration 
$$C^k(X) = \mathcal{G}^{-k} C^k(X) \supset \mathcal{G}^{-k+1} C^k(X) \supset \cdots \supset \mathcal{G}^0 C^k(X) \supset \mathcal{G}^1 C^k(X) = 0$$
on $C^*(X)$. Moreover, the contravariant functoriality of the dual geometric filtration allows the action of $G$ on $X$ to induce on action on the filtered cochain complex $\mathcal{G}^{\bullet} C^*(X)$, which make it into an object of $\mathfrak{C}^G$ :

\begin{de} \label{defdugeofilact} Let $X$ be a real algebraic $G$-variety. We denote by ${}^G\!\mathcal{G}^{\bullet} C^*(X)$ (or simply $\mathcal{G}^{\bullet} C^*(X)$ when the context is clear) the filtered cochain complex $\mathcal{G}^{\bullet} C^*(X)$ equipped with the induced action of $G$. We call the functor
$${}^G\!\mathcal{G}^{\bullet} C^* :  \mathbf{Sch}_c^G(\mathbb{R}) \rightarrow \mathfrak{C}^G$$
the dual geometric filtration with action of $G$.
\end{de}

In $\cite{LP}$ is proved that the dual geometric filtration realizes the cohomological weight complex and that it is unique up to filtered quasi-isomorphism. Recall that any bounded cochain complex equipped with a bounded decreasing filtration induces a second quadrant spectral sequence $(E_r)_{r \geq 0}$, which converges to the cohomology of the complex. A filtered quasi-isomorphism between two such filtered complexes is a filtered morphism inducing an isomorphism at the level $E_1$ of the induced spectral sequences.

We are going to prove that the dual geometric filtration with action of $G$ is unique up to equivariant filtered quasi-isomorphism. We actually show the existence and uniqueness, both up to equivariant filtered quasi-isomorphism, of a cohomological weight complex with action of $G$ on the category of real algebraic $G$-varieties, using Th\'eor\`eme 3.5 of \cite{Pri-CA}, which is a version with action of Th\'eor\`eme 2.2.2 of \cite{GNA}. We then prove that the dual geometric filtration with action of $G$ realizes the cohomological weight complex with action of $G$.

First denote by $H o \, \mathfrak{C}^G$ the localization of the category $\mathfrak{C}^G$ with respect to equivariant filtered quasi-isomorphisms (we also call such morphisms quasi-isomorphism of $\mathfrak{C}^G$) and, if $(K^*, \delta)$ is a cochain $G$-complex, denote by ${}^G\!F_{can} K^*$ the canonical filtration of $K^*$ (see for instance Definition 3.3 of \cite{LP}) equipped with the induced action of $G$. We have the following result, which is the cohomological counterpart of Th\'eor\`eme 3.7 of \cite{Pri-CA} :

\begin{theo} \label{wcactcohom}The contravariant functor
$${}^G\!F_{can} C^* : \mathbf{V}^G(\mathbb{R}) \rightarrow H o \, \mathfrak{C}^G~;~ M \mapsto {}^G\!F_{can} C^*(M)$$
extends to a contravariant functor
$${}^G\!\mathcal{W} C^* : \mathbf{Sch}_c^G(\mathbb{R}) \rightarrow H o \, \mathfrak{C}^G$$
verifying the two following conditions :

\begin{enumerate}

\item For an acyclic square
\begin{equation*} \label{acyclic_square}
\begin{array}{ccc}
\widetilde Y & \stackrel{j}{\hookrightarrow} & \widetilde X\\
 ~ \downarrow {\scriptstyle \pi} & & ~ \downarrow {\scriptstyle \pi}\\
Y & \stackrel{i}{\hookrightarrow} & X
\end{array} \tag{2.1}
 \end{equation*}
in $\mathbf{Sch}_c^G(\mathbb{R})$, the simple filtered complex of the diagram
$$\begin{array}{ccc}
{}^G\!\mathcal{W} C^*(\widetilde Y) & \stackrel{j^*}{\longleftarrow} & {}^G\!\mathcal{W} C^* (\widetilde X)\\
\uparrow_{\pi^*} & & \uparrow_{\pi^*}\\
{}^G\!\mathcal{W} C^*(Y) & \stackrel{i^*}{\longleftarrow} & {}^G\!\mathcal{W} C^*(X)
\end{array}$$
is acyclic, i.e. isomorphic in $H o \, \mathfrak{C}^G$ to the zero complex.\\

\item For an equivariant closed inclusion $Y \hookrightarrow X$, the simple filtered complex of the diagram
$${}^G\!\mathcal{W}C^*(Y) \longleftarrow {}^G\!\mathcal{W}C^*(X)$$
is isomorphic in  $H o \, \mathfrak{C}^G$ to ${}^G\!\mathcal{W}C^*(X \setminus Y)$.\\
\end{enumerate}

Such a functor ${}^G\!\mathcal{W}C^*$ is unique up to a unique equivariant filtered quasi-isomorphism and we call it the cohomological weight complex with action of $G$.
\end{theo}

\begin{rem}
\begin{itemize}
	\item An acyclic square in $\mathbf{Sch}_c^G(\mathbb{R})$ is a commutative diagram (\ref{acyclic_square}) of objects and morphisms of $\mathbf{Sch}_c^G(\mathbb{R})$ such that $i$ is an equivariant inclusion of a closed subvariety, $\widetilde{Y} = \pi^{-1}(Y)$ and the restriction $\pi : \widetilde{X} \setminus \widetilde{Y} \rightarrow X \setminus Y$ is an equivariant isomorphism.
	\item For the definition of the simple complex associated to a cubical diagram of cochain filtered complexes, we refer to \cite{LP} Definition 3.1. Notice that, if the cochain complexes are elements of $\mathfrak{C}^G$, the associated simple complex can be naturally equipped with the induced action of $G$, considering the diagonal action on direct sums.
\end{itemize}
\end{rem} 
 
\begin{proof}[Proof (of theorem \ref{wcactcohom})] As in the proof of Th\'eor\`eme 3.7 in \cite{Pri-CA}, we show the existence of the functor ${}^G\!\mathcal{W} C^* : \mathbf{Sch}_c^G(\mathbb{R}) \rightarrow H o \, \mathfrak{C}^G$ by using the functoriality of the cohomological weight complex $\mathcal{W} C^* : \mathbf{Sch}_c(\mathbb{R}) \rightarrow H o \, \mathfrak{C}$ of \cite{LP} (Theorem 3.4) : in particular, if $X$ is a real algebraic $G$-variety, ${}^G\!\mathcal{W} C^*(X)$ denotes its cohomological weight complex $\mathcal{W} C^*(X)$ equipped with the action of $G$ induced by (contravariant) functoriality. 

Similarly to the proof of Th\'eor\`eme 3.7 of \cite{Pri-CA}, the uniqueness is then given by Th\'eor\`eme 3.5 of \cite{Pri-CA}, which is a version with action of the extension criterion Th\'eor\`eme 2.2.2 of \cite{GNA} : the category $\mathfrak{C}^G$ is a category of cohomological descent, the functor ${}^G\!F_{can} C^* : \mathbf{V}^G(\mathbb{R}) \rightarrow H o \, \mathfrak{C}^G$ is $\Phi$-rectified (since it is defined on the category $\mathfrak{C}^G$) and it verifies conditions (F1) and (F2). To prove the last assertion, we can use the arguments of $\cite{LP}$ Proof of Theorem 3.4 showing that the functor $F_{can} C^* : \mathbf{V}(\mathbb{R}) \rightarrow H o \, \mathfrak{C}$ verifies the conditions (F1) and (F2) : these arguments remain valid when we consider actions of $G$ on the considered objects and morphisms.
\end{proof}

\begin{rem} 
\begin{itemize}
	\item If $X$ is a real algebraic $G$-variety, we have an isomorphism $H^*({}^G\!\mathcal{W} C^*(X)) \cong H^*(X)$, by Proposition 3.7 of \cite{LP}), and this isomorphism is furthermore equivariant. Indeed, there is actually an isomorphism of functors between the functors $\phi \circ \mathcal{W} C^*$ and $C^*(\cdot)$ : see the proof of Proposition 3.7 of \cite{LP} and see also Remarque 3.9 of \cite{Pri-CA}. 
	\item The cohomological weight filtration $\mathcal{W}^{\bullet}$, induced by $\mathcal{W} C^*$ on the cohomology of real algebraic varieties, as well as the associated cohomological weight spectral sequence  (see Corollary 3.8 of \cite{LP}), can be both equipped with the induced actions of $G$.
	\item If $X$ is a compact nonsingular $G$-variety, its cohomological weight complex with action ${}^G\!\mathcal{W} C^*(X)$ is quasi-isomorphic in $\mathfrak{C}^G$ to ${}^G\!F_{can} C^*(X)$ (we can adapt the proof of Proposition 3.11 of \cite{LP} to our framework with action).  
\end{itemize}

\end{rem}

Now, we show that the dual geometric filtration with action realizes the cohomological weight complex with action. As a consequence, we obtain the uniqueness of ${}^G\!\mathcal{G}^{\bullet} C^*$ up to equivariant filtered quasi-isomorphism.
wcactcohom dgcwc
\begin{prop} \label{dgcwc}
The dual geometric filtration with action ${}^G\!\mathcal{G}^{\bullet} C^* :  \mathbf{Sch}_c^G(\mathbb{R}) \rightarrow \mathfrak{C}^G$ induces the cohomological weight complex with action ${}^G\!\mathcal{W} C^* : \mathbf{Sch}_c^G(\mathbb{R}) \rightarrow H o \, \mathfrak{C}^G$. 
\end{prop}

\begin{proof}
We denote again by ${}^G\!\mathcal{G}^{\bullet} C^*$ the functor $\mathbf{Sch}_c^G(\mathbb{R}) \rightarrow H o \, \mathfrak{C}^G$ obtained by composing the the functor ${}^G\!\mathcal{G}^{\bullet} C^* :  \mathbf{Sch}_c^G(\mathbb{R}) \rightarrow \mathfrak{C}^G$ with the localization $\mathfrak{C}^G \rightarrow H o \, \mathfrak{C}^G$. 

First, this functor verifies the conditions 1 and 2 of theorem \ref{wcactcohom}. Indeed, the morphisms of the short exact sequences of complexes of Lemma 4.2 of \cite{LP} are equivariant if we consider $G$-varieties.

Secondly, we show that ${}^G\!\mathcal{G}^{\bullet} C^*$ verifies the extension property as well. Let $X$ be a nonsingular projective $G$-variety. The (equivariant) inclusion of the geometric filtration $\mathcal{G}_{\bullet} C_*(X)$ in the canonical filtration $F^{can}_{\bullet} C_*(X)$ induces an equivariant morphism between the dual canonical filtration and the dual geometric filtration
$$(F_{can})^{\bullet}_{\vee} C^*(X) \rightarrow \mathcal{G}^{\bullet} C^*(X)$$
which is, because $X$ is compact nonsingular, a filtered quasi-isomorphism (see the proof of Proposition 4.3 of \cite{LP}), that is a quasi-isomorphism of $\mathfrak{C}^G$. On the other hand, there is an (equivariant) inclusion of the canonical filtration $F_{can}^{\bullet} C^*(X)$ in $(F_{can})^{\bullet}_{\vee} C^*(X)$ which is also a quasi-isomorphism of $\mathfrak{C}^G$. As a consequence, $\mathcal{G}^{\bullet} C^*(X)$ and $F_{can}^{\bullet} C^*(X)$ are isomorphic in $H o \, \mathfrak{C}^G$. 
\end{proof}

As for the homological counterpart of \cite{Pri-EWF}, the dual geometric filtration will induce the cohomological equivariant weight filtration, that we construct in the next section.

\section{Cohomological equivariant weight filtration for real algebraic $G$-varieties} \label{seccohomequivwf}

Let $G$ be a finite group.

Similarly to what we did in \cite{Pri-EWF}, we construct a weight filtration on the equivariant cohomology of real algebraic $G$-varieties. The equivariant cohomology we consider is defined by J. van Hamel in \cite{VH} and coincide with the classical equivariant cohomology (see for instance \cite{Bro}) for compact varieties. It is a mix of cohomology with compact supports and group cohomology.

Precisely, we define below a functor on bounded cochain $G$-complexes which computes this equivariant cohomology, then extend it to the category of filtered bounded cochain $G$-complexes and finally apply it to the dual geometric filtration with action.

We will then focus on the induced spectral sequences, which contain rich information about the equivariant geometry of real algebraic $G$-varieties.

\subsection{The functor $L^*$}

We refer to \cite{Bro} and \cite{CTVZ} for background about group cohomology with coefficients in a module (see also the first part of section 3.1 of \cite{Pri-EWF}). In the following definition, we consider a functor $L^*$ that we use to define the cohomology of the group $G$ in a bounded cochain $G$-complex ; if $X$ is a real algebraic $G$-variety, the equivariant cohomology of $X$ will be for us the cohomology of the complex $L^*(C^*(X))$.

First denote by $\mathfrak{D}_+$ the category of bounded below cochain complexes of $\mathbb{Z}_2$-vector spaces, and by $H o \, \mathfrak{D}_+$ its localization with respect to quasi-isomorphisms. 

\begin{defprop} \label{defproplcohom} Let $K^*$ be in $\mathfrak{D}^G$. Consider $... \rightarrow F_2 \xrightarrow{\Delta_2} F_1 \xrightarrow{\Delta_1} F_0 \rightarrow \mathbb{Z} \rightarrow~0$ a resolution of $\mathbb{Z}$ by projective $\mathbb{Z}[G]$-modules.

We define the cochain complex $L^*(K^*)$ of $\mathfrak{D}_+$ to be the total complex associated to the double complex
$$(\Hom_G(F_{p},K^q))_{p,q \in \mathbb{Z}}.$$
The operation which associates to a complex $K^*$ of $\mathfrak{D}^G$ the complex $L^*(K^*)$ of $\mathfrak{D}_+$ is a covariant functor.
\end{defprop}

If $K^*$ is in $\mathfrak{D}^G$, denote by $H^*(G, K^*)$ the cohomology of $L^*(K^*)$. Since this last cochain complex is the total complex of a double complex, we have the following two spectral sequences which both converge to $H^*(G, K^*)$ :
$$\left. \begin{matrix}
{}_{I}\!E_2^{p,q} & = & H^{p}(G,H^q(K^*)) \\
{}_{II}\!E_1^{p,q} & = & H^{p}(G,K^q)
\end{matrix} \right\} \Longrightarrow H^{p+q}(G,K^*).$$ 
Considering the first spectral sequence (called the Hochschild-Serre spectral sequence associated to $G$ and $K^*$), since group cohomology with coefficients in a module is independent of the considered projective resolution, we can claim that so do $H^*(G, K^*)$. We then call $H^*(G, K^*)$ the group cohomology of $G$ with coefficients in the cochain complex $K^*$. In particular, the composition of the functor $L^*$ with the localization $\mathfrak{D}_+ \rightarrow H o \, \mathfrak{D}_+$ is also independent of the considered projective resolution.  

Furthermore, this functor $L^*$ preserves quasi-isomorphisms of $\mathfrak{D}^G$ : if $f : K^* \rightarrow M^*$ is an equivariant quasi-isomorphism, then it induces an isomorphism from the level ${}_{I}\!E_2$ of the induced Hochschild-Serre spectral sequences and therefore between the cohomologies of G with coefficients in $K^*$ and $M^*$.

We also denote by $L^*$ the induced functor $H o \, \mathfrak{D}^G \rightarrow H o \, \mathfrak{D}_+$.
\\

Now, we extend the functor $L^*$ to the categories of filtered cochain complexes $\mathfrak{C}^G$ and $\mathfrak{C}_+$, where $\mathfrak{C}_+$ denotes the category of bounded below cochain complexes of $\mathbb{Z}_2$-vector spaces equipped with a decreasing bounded filtration, and morphisms of filtered complexes.

\begin{de} Let $(K^*, J^{\bullet})$ be in $\mathfrak{C}^G$. We define an induced bounded decreasing filtration $\mathfrak{J}^{\bullet}$ on $L^*(K^*)$ by setting
$$\mathfrak{J}^{\alpha} L^k(K^*) := L^k(\mathfrak{J}^{\alpha} K^*).$$
\end{de}

As in \cite{Pri-EWF} Proposition 3.7, we can show that, in $H o \, \mathfrak{C}_+$, the couple $(L^*(K^*), \mathfrak{J}^{\bullet})$ is independent of the considered projective resolution, and, as in \cite{Pri-EWF} Proposition 3.8, that the induced functor $L^* : \mathfrak{C}^G \rightarrow H o \, \mathfrak{C}_+$ preserves filtered quasi-isomorphisms. 

This induces a well-defined functor $L^* : H o \, \mathfrak{C}^G \rightarrow H o \, \mathfrak{C}_+$.

\subsection{Cohomological equivariant weight complex, spectral sequence, filtration}

We begin this part by the definition of the equivariant cohomology of real algebraic $G$-varieties, for which we use the functor $L^*$ :

\begin{de} Let $X$ be a real algebraic $G$-variety. Denote $C^*_G(X) := L^*(C^*(X))$ (considered in $H o \, \mathfrak{D}_+$, or in $\mathfrak{D}_+$ if we fix a resolution of $\mathbb{Z}$ by projective $\mathbb{Z}[G]$-modules). We call
$$H^*(X ; G) := H^*(C^*_G(X))$$
the equivariant cohomology of $X$.
\end{de} 

\begin{rem} \label{HochSer}
The Hochschild-Serre spectral sequence
\begin{equation*} \label{HochSerss}
{}_{I}\!E_2^{p,q} = H^{p}(G,H^q(X)) \Rightarrow H^{p+q}(X ; G)
\tag{3.1}
\end{equation*}
allows to interpret the equivariant cohomology as a mix of cohomology with compact supports and group cohomology, involving the very geometry of the action of the group on the variety. This equivariant cohomology is the same as the one defined in \cite{VH} Chapter III, at least for compact real algebraic $G$-varieties. Notice that if the group $G$ is trivial, the equivariant cohomology coincide with the cohomology with compact supports (in this case, we consider the trivial resolution $... \rightarrow 0 \xrightarrow{} 0 \xrightarrow{} \mathbb{Z} \xrightarrow{id} \mathbb{Z} \rightarrow 0$).  	
\end{rem}

\begin{ex} \label{excircle1} Consider the sphere $S^1$ given by the equation $x^2 + y^2 = 1$ in $\mathbb{R}^2$ and suppose it equipped with the action of the group $G := \mathbb{Z}/2\mathbb{Z}$ given by the involution $\sigma : (x,y) \mapsto (-x,y)$. We use the Hochschild-Serre spectral sequence to compute the equivariant cohomology of $S^1$ (see also Examples 3.3 and 3.13 of \cite{Pri-EWF}).

The page ${}_{I}\!E_2$ is 
$$\begin{array}{ccccc}
\mathbb{Z}_2 \overline{[S^1]}^{\vee} &\mathbb{Z}_2 \overline{[S^1]}^{\vee} & \cdots & \mathbb{Z}_2 \overline{[S^1]}^{\vee}  & \cdots \\
\mathbb{Z}_2 \overline{[\{p_1\}]}^{\vee} &\mathbb{Z}_2 \overline{[\{p_1\}]}^{\vee} & \cdots & \mathbb{Z}_2 \overline{[\{p_1\}]}^{\vee}  & \cdots  
\end{array}
$$
where $p_1$ denotes the point of coordinates $(0 , 1)$ and $\overline{c}^{\vee}$ denotes the linear map $H^k(S^1) = (H_k(S^1))^{\vee} \rightarrow \mathbb{Z}_2$ which associates to the homology class $\overline{c}$ the value $1$. 

We compute the image of the cohomology class $\overline{[S^1]}^{\vee}$ by the differential $d_2$. For this sake, we first represent $\overline{[S^1]}^{\vee}$ by the linear map $\varphi : C_1(S^1) \rightarrow \mathbb{Z}_2$ defined as follows. Suppose that $A$ is a connected one-dimensional closed semialgebraic subset of $S^1$ with $A \neq S^1$, then
$$\varphi([A]) := \begin{cases}
1 & \mbox{ if $p_1 \in \partial A$ and $A \cap \{x < 0\} \neq \emptyset$,} \\
0 & \mbox{ ~otherwise.}
\end{cases}$$
In particular, $\varphi([S^1]) = \varphi([S^1 \cap \{x \leq 0\}]) + \varphi([S^1 \cap \{x \geq 0\}]) = 1+0 = 1$.
\\
We then apply $1 + \sigma$ to $\varphi$ : by definition, $(1 + \sigma) \cdot \varphi = \varphi + \varphi \circ \sigma$ and, if $A$ verifies the same assumptions as above,
$$(\varphi + \varphi \circ \sigma)([A]) = \begin{cases}
1 & \mbox{ if $p_1 \in \partial A$,} \\
0 & \mbox{ ~if $p_1 \notin \partial A$.}
\end{cases}$$
Now, we notice that $(1 + \sigma) \cdot \varphi = \delta_0(\psi) = \psi \circ \partial_1$ where, if $p$ is a point of $S^1$, 
$$\psi([\{p\}]) := \begin{cases}
1 & \mbox{ if $p = p_1$,} \\
0 & \mbox{ ~if $p \neq p_1$,}
\end{cases}$$
and finally, $(1 + \sigma) \cdot \psi = \psi + \psi \circ \sigma \equiv 0$, since $p_1$ is a fixed point under the action of $G$ on $S^1$. 

As a consequence, $d_2(\overline{[S^1]}^{\vee}) = 0$ and ${}_{I}\!E_2 = {}_{I}\!E_{\infty}$. As a conclusion, we obtain
$$H^k(S^1; G) = \begin{cases}
			0 & \mbox{ if $k < 0$,} \\
			\mathbb{Z}_2 & \mbox{ if $k = 0$,} \\
			\mathbb{Z}_2 \oplus \mathbb{Z}_2 & \mbox{ if $k \geq 1$.}
			\end{cases}
$$
\end{ex}

\begin{rem} \label{remnodualg} Since, under the same assumptions, 
$$H_k(S^1; G) = \begin{cases}
			\mathbb{Z}_2 \oplus \mathbb{Z}_2 & \mbox{ if $k \leq 0$,} \\
			\mathbb{Z}_2 & \mbox{ if $k = 1$,} \\
			0 & \mbox{ if $k > 1$.}
			\end{cases}
$$
we can remark that, in the general case, there is no ``classical'' duality between equivariant cohomology and equivariant homology such as for cohomology with compact supports and Borel-Moore homology (see Proposition 2.2 of \cite{LP}).
\end{rem} 

\begin{ex}
Keep the same hypothesis as in example \ref{excircle1} and suppose the action is now given by the free involution $\sigma : (x,y) \mapsto (-x,-y)$. Using the same notation as above, we have
$$(\varphi + \varphi \circ \sigma)([A]) = \begin{cases}
1 & \mbox{ if $p_1 \in \partial A$ and $A \cap \{x < 0\} \neq \emptyset$, or $p_2 \in \partial A$ and $A \cap \{x > 0\} \neq \emptyset$,} \\
0 & \mbox{ ~otherwise}
\end{cases}$$
where $p_2$ is the point of coordinates $(0, -1)$.
\\
Therefore, $(1 + \sigma) \cdot \varphi = \delta_0(\psi) = \psi' \circ \partial_1$ where, if $p$ is a point of $S^1$, 
$$\psi'([\{p\}]) := \begin{cases}
1 & \mbox{ if $p \in S^1 \cap \{x < 0\}$ or $p = p_2$,} \\
0 & \mbox{ if $p \in S^1 \cap \{x > 0\}$ or $p = p_1$.}
\end{cases}$$
and $(1 + \sigma) \cdot \psi' = \psi' + \psi' \circ \sigma : [\{p\}] \mapsto 1$.

Consequently, $d_2(\overline{[S^1]}^{\vee}) = \overline{[\{p_1\}]}^{\vee}$ and the Hochschild-Serre spectral sequence degenerates at page ${}_{I}\!E_3$ :
$$\begin{array}{cccccc}
0 & 0 & 0 & \cdots & 0  & \cdots \\
\mathbb{Z}_2 \overline{[\{p_1\}]}^{\vee} &\mathbb{Z}_2 \overline{[\{p_1\}]}^{\vee} & 0 & \cdots & 0  & \cdots  
\end{array}
$$
and we obtain
$$H^k(S^1; G) = \begin{cases}
			0 & \mbox{ if $k < 0$,} \\
			\mathbb{Z}_2 & \mbox{ if $k = 0$ or $1$,} \\
			0 & \mbox{ if $k > 1$.}
			\end{cases}
$$
\end{ex} 
 
We then consider the dual geometric filtration. We can apply the functor $L^*$, extended to the categories of filtered cochain complexes, to induce a filtration on $C^*_G(\cdot)$ :

\begin{de} If $X$ is a real algebraic $G$-variety, we denote 
$$\Lambda^{\bullet} C^*_G(X) := L^*(\mathcal{G}^{\bullet} C^*(X))$$
and we call this filtered complex of $H o \, \mathfrak{C}_+$ (or $\mathfrak{C}_+$ if we fix a projective resolution of $\mathbb{Z}$ by projective $\mathbb{Z}[G]$-modules) the cohomological equivariant geometric filtration of $X$.
\end{de}  

The operation which associates to any real algebraic $G$-variety its cohomological equivariant geometric filtration is a contravariant functor, since it is the composition of the functors $\mathcal{G}^{\bullet} C^* :  \mathbf{Sch}_c^G(\mathbb{R}) \rightarrow \mathfrak{C}^G$ and $L^*: \mathfrak{C}^G \rightarrow (H o) \mathfrak{C}_+$.
\\

We are going to show that the functor $\Lambda^{\bullet} C^*_G : \mathbf{Sch}_c^G(\mathbb{R}) \rightarrow H o \, \mathfrak{C}_+$ is unique up to filtered quasi-isomorphism (of $\mathfrak{C}_+$) in a way similar to theorem \ref{wcactcohom} and proposition \ref{dgcwc} (see also the homological counterpart Theorem 3.16 of \cite{Pri-EWF}).

For $K^*$ a bounded cochain $G$-complex, we denote $\mathcal{F}^{\bullet}_{can} L^*(K^*) := L^*(F_{can}^{\bullet} K^*)$.

\begin{theo} \label{unicoequivwc} The contravariant functor
$$\mathcal{F}_{can} C_G^* : \mathbf{V}^G(\mathbb{R}) \rightarrow H o \, \mathfrak{C}_+~;~ M \mapsto \mathcal{F}_{can} C_G^*(M)$$
extends to a contravariant functor
$$\Omega C_G^* : \mathbf{Sch}_c^G(\mathbb{R}) \rightarrow H o \, \mathfrak{C}_+$$
verifying the two following conditions :

\begin{enumerate}

\item For an acyclic square
\begin{equation*} \label{acyclic_square2}
\begin{array}{ccc}
\widetilde Y & \stackrel{j}{\hookrightarrow} & \widetilde X\\
 ~ \downarrow {\scriptstyle \pi} & & ~ \downarrow {\scriptstyle \pi}\\
Y & \stackrel{i}{\hookrightarrow} & X
\end{array}
 \end{equation*}
in $\mathbf{Sch}_c^G(\mathbb{R})$, the simple filtered complex of the diagram
$$\begin{array}{ccc}
\Omega C_G^*(\widetilde Y) & \stackrel{j^*}{\longleftarrow} & \Omega C_G^*(\widetilde X)\\
\uparrow_{\pi^*} & & \uparrow_{\pi^*}\\
\Omega C_G^*(Y) & \stackrel{i^*}{\longleftarrow} &\Omega C_G^*(X)
\end{array}$$
is acyclic, i.e. isomorphic in $H o \, \mathfrak{C}_+$ to the zero complex.\\

\item For an equivariant closed inclusion $Y \hookrightarrow X$, the simple filtered complex of the diagram
$$\Omega C_G^*(Y) \longleftarrow \Omega C_G^*(X)$$
is isomorphic in  $H o \, \mathfrak{C}_+$ to $\Omega C_G^*(X \setminus Y)$.\\
\end{enumerate}

Such a functor $\Omega C_G^*$ is unique up to a unique filtered quasi-isomorphism and we call it the cohomological equivariant weight complex.
\end{theo}

\begin{proof} {\it Existence :} The composition of the functor ${}^G\!\mathcal{W} C^* : \mathbf{Sch}_c^G(\mathbb{R}) \rightarrow H o \, \mathfrak{C}^G$ with the functor $L^* : H o \, \mathfrak{C}^G \rightarrow H o \, \mathfrak{C}_+$ verifies the extension, acyclicity and additivity properties, because so do the functor ${}^G\!\mathcal{W} C^*$ (theorem \ref{wcactcohom}), and because the functor $L^*$ preserves filtered quasi-isomorphisms and commutes with the operation associating to a cubical diagram in $\mathfrak{C}^G$ its simple filtered diagram (it is the cohomological counterpart of Proposition 3.10 of \cite{Pri-EWF}, obtained by a direct computation). We denote this functor $\mathbf{Sch}_c^G(\mathbb{R}) \rightarrow H o \, \mathfrak{C}_+$ by $\Omega C_G^*$.

{\it Uniqueness :} The uniqueness of the cohomological equivariant weight complex with these properties can be obtained in the same way as the uniqueness of the equivariant homological weight complex (proof of Theorem 3.16 of \cite{Pri-EWF}) : use Th\'eor\`eme 3.5 of \cite{Pri-CA} applied to the category of cohomological descent $\mathfrak{C}_+$ (Propri\'et\'e (1.7.5) of \cite{GNA}) and the functor $\mathcal{F}_{can} C_G^* : \mathbf{V}^G(\mathbb{R}) \rightarrow H o \, \mathfrak{C}_+$.
\end{proof}

\begin{rem} If $X$ is a compact nonsingular $G$-variety, its cohomological equivariant weight complex $\Omega C_G^*(X)$ is quasi-isomorphic in $\mathfrak{C}_+$ to $\mathcal{F}_{can} C_G^*(X)$ (because in this case ${}^G\!\mathcal{W} C^*(X)$ is quasi-isomorphic in $\mathfrak{C}^G$ to ${}^G\!F_{can} C^*(X)$).
\end{rem}

Since, by proposition \ref{dgcwc}, the dual geometric filtration with action realizes the cohomological weight complex with action (and because the functor $L^*$ preserves filtered quasi-isomorphisms), we obtain :

\begin{prop} \label{propcoequivgeowc} The cohomological equivariant geometric filtration $\Lambda^{\bullet} C^*_G : \mathbf{Sch}_c^G(\mathbb{R}) \rightarrow H o \, \mathfrak{C}_+$ induces the cohomological equivariant weight complex $\Omega C_G^* : \mathbf{Sch}_c^G(\mathbb{R}) \rightarrow H o \, \mathfrak{C}_+$.
\end{prop}

If $X$ is a real algebraic $G$-variety, we denote by ${}^G\!E_*(X)$ the spectral sequence induced by the filtered cohomological equivariant weight complex : it is well-defined from page $E_1$ and coincide, from page $E_1$, with the spectral sequence induced by the cohomological equivariant geometric filtration. We call it the cohomological equivariant weight spectral sequence of $X$. It converges to the equivariant cohomology $H^*(X ; G)$ of $X$ and we denote by $\Omega$ the decreasing filtration induced on $H^*(X ; G)$. 

As in \cite{MCP}, \cite{LP} and \cite{Pri-EWF}, we reindex the cohomological equivariant weight spectral sequence by setting ${}^G\!\widetilde{E}^{p,q}_r = {}^G\!E^{-q,p+2q}_{r-1}$. We can then read the acyclicity and additivity of the cohomological equivariant weight complex on the rows of the page two of this reindexed spectral sequence~: for instance, if we have an acyclic square (\ref{acyclic_square}), we have, for all $q \in \mathbb{Z}$, a long exact sequence 
$$\cdots \rightarrow {}^G\!\widetilde{E}^{p,q}_2(X) \rightarrow {}^G\!\widetilde{E}^{p,q}_2(Y) \oplus {}^G\!\widetilde{E}^{p,q}_2(\widetilde{X}) \rightarrow {}^G\!\widetilde{E}^{p,q}_2(\widetilde{Y}) \rightarrow {}^G\!\widetilde{E}^{p+1,q}_2(X) \rightarrow \cdots$$
As in \cite{Pri-EWF}, Proposition 3.17, we can also express the page two of ${}^G\!E_*$ as the cohomology of the group $G$ with coefficients in the cohomological weight spectral sequence in the following sense~: for all $p,q \in \mathbb{Z}$,
$${}^G\!\widetilde{E}_2^{p,q} = H^p\left(G, \widetilde{E}_1^{*,q}\right).$$
This leads in particular to consider the following (Hochschild-Serre) spectral sequences
$${}^{q}_{I}\!E_2^{\alpha,\beta} = H^{\alpha}\left(G, \widetilde{E}_2^{\beta, q}\right) \Rightarrow H^{\alpha + \beta}\left(G, \widetilde{E}_1^{*,q}\right) = {}^G\!\widetilde{E}_2^{\alpha + \beta,q},$$
which allows to obtain  the following bounds on the cohomological equivariant weight spectral sequence and filtration :

\begin{prop}
Let $X$ be a real algebraic $G$-variety of dimension $d$. For all $r \geq 2$, $p,q \in \mathbb{Z}$, if $^G\!\widetilde{E}_r^{p,q} \neq 0$ then $0 \leq q \leq d$ and $p \geq 0$, and, for all $k \in \mathbb{Z}$, we have the inclusions
$$H^k(X ; G) = \Omega^{-d} H^k(X ; G) \supset \Omega^{-d+1} H^k(X ; G) \supset \cdots \supset \Omega^0 H^k(X ; G) \supset \Omega^1 H^k(X ; G) = 0.$$ 
\end{prop}

\begin{proof}
For the first part, for $p, q \in \mathbb{Z}$, consider the spectral sequence
$${}^{q}_{I}\!E_2^{\alpha,\beta} = H^{\alpha}\left(G, \widetilde{E}_2^{\beta, q}\right) \Rightarrow {}^G\!\widetilde{E}_2^{\alpha + \beta,q},$$
and use the fact that, for $\alpha < 0$, $H^{\alpha}(G, \cdot) = 0$ and that, for all $\beta \in \mathbb{Z}$, $\widetilde{E}_2^{\beta, q} \neq 0$ implies $\beta \geq 0$, $q \geq 0$ and $\beta + q \leq d$ (see \cite{LP}).

Finally, we prove that $H^k(X ; G) = \Omega^{-d} H^k(X ; G)$ and $\Omega^1 H^k(X ; G) = 0$ by using the equalities
$$\Omega^l H^k(X;G) = \bigoplus_{m \leq -l} {}^G\!\widetilde{E}_{\infty}^{k- m, m}.$$

\end{proof}

\begin{rem} \label{remaddcohom} As in the homological framework of \cite{Pri-EWF},
\begin{itemize}
	\item the cohomological equivariant weight spectral sequence of a compact nonsingular real algebraic $G$-variety coincides with its Hochschild-Serre spectral sequence (\ref{HochSerss}) (see Proposition 3.23 of \cite{Pri-EWF}),
	\item we can extract additivities expressed in terms of the spectral sequences
	$${}^{~q}_{II}\!E_1^{\alpha,\beta} = H^{\alpha}\left(G, \widetilde{E}_1^{\beta, q}\right)$$
induced by the dual geometric filtration (see section 4 of \cite{Pri-EWF}),
	\item if $G$ is an odd-order group, we have, for any cochain $G$-complex $K^*$, $L^*(K^*) = (K^*)^G$ and therefore $\Lambda^{\bullet} C^*_G$ and ${}^G\!E_* = (E_*)^G$ (see section 3.4 of \cite{Pri-EWF}) : as a direct consequence, in this case, we can recover the equivariant virtual Betti numbers of \cite{GF} on the rows of ${}^G\!E_2$. 
\end{itemize}

\end{rem}

\section{Products of equivariant weight filtrations} \label{sectprod}

In the last section of \cite{LP}, the K\"unneth isomorphism as well as the cup and cap products on homology and cohomology are showed to be induced from the weight spectral sequences level. In this part, we show how these products on the weight spectral sequences induce their equivariant counterparts on the equivariant weight spectral sequences.

\subsection{Equivariant K\"unneth isomorphism} \label{subsectku}

Let $G$ and $G'$ be two finite groups.

In this paragraph, we show the equivariant analog of Theorem 5.13 of \cite{LP} : if $X$ is a real algebraic $G$-variety and $Y$ is a real algebraic $G'$-variety, then the $G \times G'$-equivariant geometric filtration of the product $X \times Y$ is filtered quasi-isomorphic to the tensor product of the equivariant geometric filtrations of $X$ and $Y$. In order to prove this, we establish the equivariancy of the quasi-isomorphism of Theorem 5.13 of \cite{LP} and we apply the functor $L$, considering a particular projective resolution, namely the bar resolution.
     
We also give the cohomological counterpart of this filtered quasi-isomorphism, with respect to the cohomological equivariant geometric filtration.
\\

If $E$ and $F$ are respectively a $\mathbb{Z}[G]$-module and a $\mathbb{Z}[G']$-module, we can equip the tensor product $E \otimes_{\mathbb{Z}} F$ with an action of $G \times G'$ by setting 
$$(g,g') \cdot (x \otimes y) := g \cdot x \otimes g' \cdot y.$$
If $\mathcal{C}^G$ denotes the category of filtered bounded chain $G$-complexes of $\mathbb{Z}_2$-vector spaces, and if $(K_*, F)$ is now a filtered complex of $\mathcal{C}^G$ and $(M_*,J)$ a filtered complex of $\mathcal{C}^{G'}$, then the tensor product $(K \otimes_{\mathbb{Z}_2} M)_*$ can be equipped with an action of $G \times G'$, by considering the diagonal action of $G \times G'$. Furthermore, the induced filtration $F \otimes J$ (see \cite{LP} Definition 5.10) as well as the induced differential are equivariant with respect to this action :

\begin{lem} The filtered complex $((K \otimes_{\mathbb{Z}_2} M)_*, F \otimes J)$ can be considered as an element of $\mathcal{C}^{G \times G'}$.
\end{lem}

\begin{proof} We check that the action of $G \times G'$ commutes with the differential.
Let $n$ be an integer, $z = \sum_{i,j} x_i \otimes y_j \in (K \otimes_{\mathbb{Z}_2} M)_n$ and $(g, g') \in G \times G'$. 

If we denote by $\partial$ and $\partial'$ the respective differentials of $K_*$ and $M*$, we have
\begin{eqnarray*}
d ( (g,g') \cdot z) & = & d\left( \sum_{i,j} g \cdot x_i \otimes g' \cdot y_j \right) \\
			  & = & \sum_{i,j} \left[ \partial (g \cdot x_i)  \otimes g' \cdot y_j + g \cdot x_i  \otimes \partial'(g' \cdot y_j) \right] \\
			  & = & \sum_{i,j} \left[ g \cdot \partial (x_i)  \otimes g' \cdot y_j + g \cdot x_i  \otimes g' \cdot \partial'(y_j) \right] \\
			  & = & (g,g') \cdot d(z)
\end{eqnarray*} 
\end{proof}

As a consequence, if $X$ is a real algebraic $G$-variety and $Y$ is a real algebraic $G'$-variety, the tensor product $\mathcal{G}_{\bullet} C_*(X) \otimes \mathcal{G}_{\bullet} C_*(Y)$ is a filtered $G \times G'$-complex of $\mathcal{C}^{G \times G'}$. We show that the filtered quasi-isomorphism of Theorem 5.13 of \cite{LP} from $\mathcal{G}_{\bullet} C_*(X) \otimes \mathcal{G}_{\bullet} C_*(Y)$ to $\mathcal{G}_{\bullet} C_*(X \times Y)$ is $(G \times G')$-equivariant.

\begin{prop} \label{uequiv} The morphism of chain complexes
$$u : \mathcal{G}_{\bullet} C_*(X) \otimes \mathcal{G}_{\bullet} C_*(Y) \rightarrow \mathcal{G}_{\bullet} C_*(X \times Y)~;~c_X \otimes c_Y \mapsto c_X \times c_Y$$
is a filtered quasi-isomorphism of $\mathcal{C}^{G \times G'}$.
\end{prop}

\begin{proof} If $c$ is a semialgebraic $q$-chain of $X$ represented by a closed semialgebraic set $A$ and $c'$ is a semialgebraic $q'$-chain of $Y$ represented by a closed semialgebraic set $B$, then the product $c \times c'$ is by definition the $q+q'$-chain represented by the product $A \times B$ (see \cite{MCP} Appendix and \cite{LP} Definition 5.1). If furthermore $(g,g') \in G \times G'$, we have 
$$(g,g') \cdot c \times c' = (g,g') \cdot [A \times B] = [g \cdot A \times g' \cdot B] = g \cdot c \times g' \cdot c'.$$

Therefore, if $z = \sum_{i,j} c_i \otimes c_j' \in (C_*(X) \otimes C_*(Y))_n$,
\begin{eqnarray*}
(g,g') \cdot u(z) & = &  (g,g') \cdot \left( \sum_{i,j} c_i \times c_j' \right) \\
& = &  \sum_{i,j} g \cdot c_i \times g' \cdot c_j' \\
& = & u\left(\sum_{i,j} g \cdot c_i \otimes g' \cdot c_j' \right) \\
& = & u((g,g') \cdot z). 
\end{eqnarray*}
\end{proof}

As a direct consequence, the weight complex with action ${}^{G \times G'}\!\mathcal{W} C_*(X \times Y)$ (see \cite{Pri-CA} Th\'eor\`eme 3.5) is isomorphic to the filtered complex ${}^G\!\mathcal{W} C_*(X) \otimes {}^{G'}\!\mathcal{W} C_*(Y)$ in $H o \, \mathcal{C}^G$. Moreover, the equivariancy of the filtered quasi-isomorphism $u$ induces the equivariancy of the isomorphism on spectral sequences from level one
$$ \bigoplus_{p+s=a, \ q+t=b} E^r_{p,q}(X) \otimes E^r_{s,t}(Y) \stackrel{\sim}{\longleftarrow} E^r_{a,b} ( \mathcal{G}_{\bullet} C_*(X) \otimes \mathcal{G}_{\bullet} C_*(Y) )  \stackrel{\sim}{\longrightarrow} E^r_{a,b}(X \times Y)$$
and in particular the equivariancy of the K\"unneth isomorphism
$$\mathcal{W} H_*(X) \otimes \mathcal{W} H_*(Y) \rightarrow \mathcal{W} H_*(X \times Y)$$
(see \cite{LP} Corollary 5.14).
\\

In a second time, we apply the functor $L^{G \times G'}_*$ (see \cite{Pri-EWF} section 3.1) to the equivariant filtered quasi-isomorphism $u$. If we denote $\Lambda_{\bullet} C^G_* := L_* \circ \mathcal{G}_{\bullet} C_*$, this provides us a filtered quasi-isomorphism of $\mathcal{C}_-$
$$L_*(\mathcal{G}_{\bullet} C_*(X) \otimes \mathcal{G}_{\bullet} C_*(Y)) \rightarrow \Lambda C^{G\times G'}_*(X \times Y).$$
We are going to show that the filtered complex $L_*(\mathcal{G}_{\bullet} C_*(X) \otimes \mathcal{G}_{\bullet} C_*(Y))$ is isomorphic in $H o \, \mathcal{C}_-$ to the tensor product of the equivariant geometric filtrations $\Lambda_{\bullet} C^G_*(X) \otimes \Lambda_{\bullet} C^{G'}_*(Y)$ of $X$ and $Y$. Actually, we prove that, if we consider a particular projective resolution, these filtered complexes are isomorphic in $\mathcal{C}_-$.

The resolution we consider is the bar resolution. The definition of the bar resolution we consider can be found in \cite{CTVZ}. The important point is that, if $G$ is a finite group, all the modules of the bar resolution are finitely generated. 

\begin{de} Let $G$ be a (not necessarily) finite group and $k$ be a commutative unitary ring. If $n$ is a nonnegative integer, denote by $B_n$ the tensor product of $n+1$ copies of $k[G]$ over $k$. We make each $B_n$ into a $k[G]$-module by considering the action of $G$ on the first term $k[G]$ of $B_n$~: it is a free $k[G]$-module, generated by the elements $e \otimes g_1 \otimes \cdots \otimes g_n$.

If we define the applications
$$\epsilon : B_0 = k[G] \rightarrow k ~;~ \sum_{g \in G} a_g g \mapsto  \sum_{g \in G} a_g$$
and, for $n \geq 1$,
$$\partial_n : \begin{array}{rcl}
			B_n & \rightarrow & B_{n-1} \\
			g_0 \otimes \cdots \otimes g_n & \mapsto & \displaystyle{\sum_{i = 0}^{n-1} (-1)^i g_0 \otimes \cdots \otimes g_{i-1} \otimes g_i g_{i+1} \otimes g_{i+2} \otimes \cdots \otimes g_n + (-1)^n g_0 \otimes \cdots \otimes g_{n-1}},
		     \end{array}
$$
we can form a free resolution
$$\cdots \rightarrow B_2 \xrightarrow{\partial_2} B_1 \xrightarrow{\partial_1} B_0 \xrightarrow{\epsilon} k \rightarrow 0$$
of $k$ by $k[G]$-modules, called the bar resolution of $k$ over $G$.
\end{de}

Now, let us come back to the finite groups $G$ and $G'$ we considered at the beginning of this section and consider the bar resolutions of $\mathbb{Z}$ over $G$ and $G'$, which we denote by $B$, resp. $B'$. The tensor product of $B$ and $B'$ is a projective resolution of $\mathbb{Z}$ by $\mathbb{Z}[G \times G']$-modules (see for instance \cite{Bro} Chapter V, Proposition (1.1)) and we are going to show the following :

\begin{prop} \label{isolprod} Let $K_*$ be a chain complex of $\mathcal{D}^G$ and $M_*$ be a chain complex of $\mathcal{D}^{G'}$. Then we have a natural isomorphism
$$L^G_{B}(K_*) \otimes L^{G'}_{B'}(M_*) \rightarrow L_{B \otimes B'}^{G \times G'}(K_* \otimes M_*)$$
of $\mathcal{D}_-$ ($\mathcal{D}^G$ and $\mathcal{D}_-$ are the homological analogues of $\mathfrak{D}^G$ and $\mathfrak{D}_+$ : see \cite{Pri-EWF}).
\end{prop}

To prove this property, we will use the fact that the resolutions $B$ and $B'$ are finitely generated together with the following :

\begin{prop}[\cite{MacLane} Chapter VI - 8, (8.10) and Proposition 8.3] \label{prophomtens} Let $B$ and $B'$ be respectively finitely generated free $\mathbb{Z}[G]$ and $\mathbb{Z}[G']$-modules and let $A$, resp. $A'$, be a $\mathbb{Z}[G]$-module, resp. a $\mathbb{Z}[G']$-module.

The canonical morphism
$$\begin{array}{ccc}
Hom_G(B,A) \otimes Hom_{G'}(B',A') & \rightarrow & Hom_{G \times G'} (B \otimes B', A \otimes A') \\
f \otimes f' & \mapsto & \{ b \otimes b' \mapsto f(b) \otimes f'(b') \}
\end{array}$$
is an isomorphism.

\end{prop}

\begin{proof}[Proof (of Proposition \ref{isolprod})] 
Let $k \in \mathbb{Z}$. Then
\begin{eqnarray*}
\left(L_*^G(K_*) \otimes L_*^{G'}(M_*)\right)_k & = & \bigoplus_{i+j = k} L_i^G(K_*) \otimes L_j^{G'}(M_*) \\
& = &  \bigoplus_{i+j = k} \left(\bigoplus_{p+q =i} Hom_G(B_{-p}, K_q) \right) \otimes \left(\bigoplus_{p'+q' =j} Hom_{G'}(B'_{-p'}, M_{q'}) \right) \\
& \stackrel{\sim}{\longrightarrow} & \bigoplus_{i+j =k,\, p+q = i, \,p'+q'=j} Hom_{G \times G'} (B_{-p} \otimes B'_{-p'}, K_q \otimes M_{q'}) \\
& = & \bigoplus_{i,p,q'} Hom_{G \times G'} (B_{-p} \otimes B'_{-(k-i-q')}, K_{i-p} \otimes M_{q'})
\end{eqnarray*}
On the other hand, we have
\begin{eqnarray*}
L_k^{G \times G'} (K_* \otimes M_*) & = & \bigoplus_{r+s = k} Hom_{G \times G'} \left( \bigoplus_{-p-p'=-r} B_{-p} \otimes B'_{-p'} ~ , \bigoplus_{q +q' = s} K_q \otimes M_{q'} \right) \\
& = & \bigoplus_{r,p,q'} Hom_{G \times G'} \left( B_{-p} \otimes B'_{-(r-p)} , K_{k-r-q'} \otimes M_{q'}\right) \\
& = & \bigoplus_{i,p,q'} Hom_{G \times G'} \left( B_{-p} \otimes B'_{-(k-i-q')} , K_{i-p} \otimes M_{q'}\right) \mbox{ (we set $i := k-r-q'+p$).}
\end{eqnarray*}
For each $k \in \mathbb{Z}$, we then denote by $\psi_k$ the isomorphism
$$\left(L_*^G(K_*) \otimes L_*^{G'}(M_*)\right)_k \stackrel{\sim}{\longrightarrow} L_k^{G \times G'} (K_* \otimes M_*)$$
and we can show by a direct computation that the morphisms $\psi_*$ commute with the differentials of the two complexes $L^G_{B}(K_*) \otimes L^{G'}_{B'}(M_*)$ and $L_{B \otimes B'}^{G \times G'}(K_* \otimes M_*)$, using the naturality of the isomorphisms 
$$Hom_G(B_{-p}, K_{i-p}) \otimes Hom_{G'}(B'_{-(k-i-q')}, M_{q'}) \rightarrow Hom_{G \times G'} (B_{-p} \otimes B'_{-(k-i-q')}, K_{i-p} \otimes M_{q'}).$$
\end{proof}

It remains to show that the above morphism $\psi_*$ is a morphism of filtered complexes :

\begin{prop} \label{prodlhom} Suppose that $K_*$ is a chain complex of $\mathcal{C}^G$ and $M_*$ is a chain complex of $\mathcal{C}^{G'}$. Then the natural isomorphism 
$$L^G_{B}(K_*) \otimes L^{G'}_{B'}(M_*) \rightarrow L_{B \otimes B'}^{G \times G'}(K_* \otimes M_*)$$
is a filtered morphism of $\mathcal{C}_-$ with respect to the induced filtrations (see \cite{Pri-EWF}, as well as \cite{LP} Definition 5.10 for the definition of the tensor product of filtered complexes that we can extend to the category $\mathcal{C}_-$).
\end{prop}

\begin{proof} Let $J_{\bullet}$ and $I_{\bullet}$ be the respective equivariant filtrations of $K_*$ and $M_*$ and denote by $\mathcal{J}_{\bullet}$ and $\mathcal{I}_{\bullet}$ the induced filtrations on $L_*^G(K_*)$ and $L_*^{G'}(M_*)$. Then, if $k, l \in \mathbb{Z}$, we have
\begin{eqnarray*}
(\mathcal{J} \otimes \mathcal{I})_l (L_*^G(K_*) \otimes L_*^G(M_*))_k & := & \bigoplus_{i+j = k} \sum_{a + b = l} \mathcal{J}_a L_i^G(K_*) \otimes \mathcal{I}_b L_j^{G'}(M_*) \\
& = & \sum_{a + b = l} \bigoplus_{i+j = k} L_i^G(J_a K_*) \otimes L_j^{G'} (I_b M_*) \\
& \stackrel{\sim}{\longrightarrow} & \sum_{a + b = l} L_k^{G \times G'}(J_a K_* \otimes I_b M_*) \\
& = &  L_k^{G \times G'}\left( \sum_{a + b = l} J_a K_* \otimes I_b M_*\right) \mbox{ ($B$ and $B'$ are finitely generated)} \\
& = & L_k^{G \times G'}\left( (J \otimes I)_l (K_* \otimes M_*)\right).
\end{eqnarray*} 

\end{proof}

We can finally apply this result to obtain the quasi-isomorphisms of $\mathcal{C}_-$ 
$$\Lambda C^{G}_*(X) \otimes \Lambda C^{G'}_*(Y) \stackrel{\sim}{\longleftarrow} L^{G \times G'}_*(\mathcal{G}_{\bullet} C_*(X) \otimes \mathcal{G}_{\bullet} C_*(Y)) \longrightarrow \Lambda C^{G\times G'}_*(X \times Y).$$
 
As a consequence :

\begin{theo} \label{thequivhomprod} Let $X$ be a real algebraic $G$-variety and $Y$ be a real algebraic $G'$-variety. There is an isomorphism of $H o \, \mathcal{C}_-$
$$\Lambda C^{G}_*(X) \otimes \Lambda C^{G'}_*(Y) \rightarrow \Lambda C^{G\times G'}_*(X \times Y).$$

Consequently, the tensor product of the equivariant weight complexes $\Omega C^{G}_*(X) \otimes \Omega C^{G'}_*(Y)$ of $X$ and $Y$ is isomorphic in $H o \, \mathcal{C}_-$ to the equivariant weight complex $\Omega C^{G \times G'}_*(X \times Y)$ of the product variety. In particular, we obtain a filtered isomorphism
$$\Omega H_*(X ; G) \otimes \Omega H_*(Y ; G') \rightarrow \Omega H_*(X \times Y ; G \times G').$$
\end{theo} 

\begin{rem} The complexes $\Omega C^{G}_*(X) \otimes \Omega C^{G'}_*(Y)$ and $\Omega C^{G \times G'}_*(X \times Y)$ are isomorphic in $H o \, \mathcal{C}_-$ regardless of the considered representative of the equivariant weight complex in $\mathcal{C}_-$ and of the considered projective resolutions, by spectral sequences arguments : see Lemma 5.11 of \cite{LP} and Proposition 3.7 of \cite{Pri-EWF}.  
\end{rem}

We end this part with the cohomological counterpart of \ref{thequivhomprod} :
\begin{theo} \label{thequivcohomprod} Let $X$ be a real algebraic $G$-variety and $Y$ be a real algebraic $G'$-variety. The complexes $\Lambda C_{G}^*(X) \otimes \Lambda C_{G'}^*(Y)$ and $\Lambda C^{G\times G'}_*(X \times Y)$ are isomorphic in $H o \, \mathfrak{C}_+$. 

As a consequence, the tensor product of the cohomological equivariant weight complexes $\Omega C_{G}^*(X) \otimes \Omega C_{G'}^*(Y)$ of $X$ and $Y$ is isomorphic in $H o \, \mathfrak{C}_+$ to the cohomological equivariant weight complex $\Omega C_{G \times G'}^*(X \times Y)$ of the product, and we obtain a filtered isomorphism
$$\Omega H^*(X ; G) \otimes \Omega H^*(Y ; G') \rightarrow \Omega H^*(X \times Y ; G \times G').$$
\end{theo} 

\begin{proof} First, we show that the cochain complexes $\mathcal{G}^{\bullet} C^*(X) \otimes \mathcal{G}^{\bullet} C^*(Y)$ and $\mathcal{G}^{\bullet} C^*(X \times Y)$ are isomorphic in $H o \, \mathfrak{C}^G$. Referring to the proof of Proposition 5.17 in \cite{LP}, we have the following quasi-isomorphisms in $\mathfrak{C}$ :
$$\mathcal{G}^{\bullet} C^*(X \times Y)  \stackrel{u^{\vee}}{\longrightarrow} (\mathcal{G}_{\bullet} C_*(X) \otimes \mathcal{G}_{\bullet} C_*(Y) )^{\vee} \stackrel{w}{\longleftarrow} \mathcal{G}^{\bullet} C^*(X) \otimes \mathcal{G}^{\bullet} C^*(Y).$$
Since the morphism $u$ is equivariant (proposition \ref{uequiv}), so is $u^{\vee}$. The morphism $w$ is equivariant as well : if $(g,g') \in G \times G'$, $\varphi \otimes \psi \in C^*(X) \otimes C^*(Y)$ and $\sum_{i,j} c_i \otimes c_j' \in C_*(X) \otimes C_*(Y)$, we have
\begin{eqnarray*}
(g,g') \cdot (w ( \varphi \otimes \psi)) \left(\sum_{i,j} c_i \otimes c_j' \right) & = & w ( \varphi \otimes \psi)\left(\sum_{i,j} g^{-1} \cdot c_i \otimes {g'}^{-1} \cdot c_j' \right) \\
& = & \sum_{i,j} \varphi\left(g^{-1} \cdot c_i \right) \psi\left( {g'}^{-1} \cdot c_j'\right) \\
& = & \sum_{i,j} \left(g \cdot \varphi(c_i )\right) \left(g' \cdot \psi(c_j')\right) \\
& = & w\left( (g,g') \cdot (\varphi \otimes \psi)\right) \left(\sum_{i,j} c_i \otimes c_j' \right)
\end{eqnarray*}

As a consequence, $\mathcal{G}^{\bullet} C^*(X) \otimes \mathcal{G}^{\bullet} C^*(Y)$ and $\mathcal{G}^{\bullet} C^*(X \times Y)$ are isomorphic in $H o \, \mathfrak{C}^G$ via the equivariant filtered quasi-isomorphisms $u^{\vee}$ and $w$ of $\mathfrak{C}^G$.
\\

We can then apply the functor $L^*_{G \times G'}$ to obtain an isomorphism of $H o \, \mathfrak{C}_+$ between the cohomological equivariant geometric filtration $\Lambda C^{G\times G'}_*(X \times Y)$ of the product  $X \times Y$ and the cochain complex $L^*_{G \times G'}\left(\mathcal{G}^{\bullet} C^*(X) \otimes \mathcal{G}^{\bullet} C^*(Y)\right)$. We then use the natural isomorphism of $\mathfrak{C}_+$
$$L_G(K^*) \otimes L_{G'}(M^*) \rightarrow L_{G \times G'}(K^* \otimes M^*)$$
for $K^* \in \mathfrak{C}^G$ and $M^* \in \mathfrak{C}^{G'}$, obtained as in proposition \ref{prodlhom} considering bar resolutions over $G$ and $G'$, to show that $L^*_{G \times G'}\left(\mathcal{G}^{\bullet} C^*(X) \otimes \mathcal{G}^{\bullet} C^*(Y)\right)$ is isomorphic to the tensor product $\Lambda C_{G}^*(X) \otimes \Lambda C_{G'}^*(Y)$ of the cohomological equivariant geometric filtrations of $X$ and $Y$.

\end{proof}

\subsection{Equivariant cup product} \label{subsectcup}

Let $G$ be a finite group and let $X$ be a real algebraic $G$-variety.

In this part, we will prove that the cup product on equivariant cohomology of real algebraic $G$-varieties, defined in \cite{VH} Chapter III - 3, is filtered with respect to the cohomological equivariant weight filtration, from the cohomological equivariant weight spectral sequence level. It will be induced by the cup product defined from the cohomological weight spectral sequence level (see \cite{LP} Propositions 5.20), through the application of the functor $L$.

An important point of the study will consist in carrying the usual properties of cup product on the equivariant cup product.
\\

In \cite{LP}, the cup product is defined in the localized category $H o \, \mathfrak{C}$ as the composition

$$\smile~: \mathcal{G}^{\bullet} C^*(X) \otimes \mathcal{G}^{\bullet} C^*(X) \xrightarrow{(u^{\vee})^{-1} \circ w} \mathcal{G}^{\bullet} C^*(X \times X) \xrightarrow{\Delta^*} \mathcal{G}^{\bullet} C^*(X)$$
where $\Delta^*$ is the morphism of $\mathfrak{C}$ induced by the diagonal map $\Delta : X \rightarrow X \times X~;~x \mapsto (x,x)$.

The isomorphism $(u^{\vee})^{-1} \circ w$ of $H o \, \mathfrak{C}$ is equivariant with respect to the actions induced by the action of the product group $G \times G$ on $X \times Y$, and induces an isomorphism
$$\Lambda C_{G}^*(X) \otimes \Lambda C_{G}^*(X) \longrightarrow \Lambda C_{G\times G}^*(X \times X)$$ 
of $H o \, \mathfrak{C}_+$ (see theorem \ref{thequivhomprod}). We are then going to show that the morphism $\Delta^*$ induces a morphism
$$\Lambda C_{G\times G}^*(X \times X) \longrightarrow \Lambda C_{G}^*(X)$$
using the functoriality of $L$ with respect to the group :
\begin{prop} \label{functgroup} Let $G'$ be a finite group and let $\varphi : G \rightarrow G'$ be a morphism of groups. Let $K^*$ be a cochain complex of $\mathfrak{D}^{G'}$, resp. $\mathfrak{C}^{G'}$. Then $K^*$ can be considered as an element of $\mathfrak{D}^{G}$, resp. $\mathfrak{C}^{G}$, via $\varphi$ (if $g \in G$ and $x \in K^k$, we set $g \cdot x := \varphi(g) \cdot x$), and $\varphi$ induces furthermore a morphism
$$T : L^*_{G'}(K^*) \longrightarrow L^*_{G}(K^*)$$
of $\mathfrak{D}_+$, resp $\mathfrak{C}_+$.
\end{prop}

\begin{proof} Let $F \xrightarrow{\epsilon} \mathbb{Z}$, resp. $F' \xrightarrow{\epsilon'} \mathbb{Z}$, be a resolution of $\mathbb{Z}$ by projective $\mathbb{Z}[G]$-modules, resp. $\mathbb{Z}[G']$-modules. There exists an augmentation-preserving $G$-chain map $\tau : F \rightarrow F'$, well-defined up to homotopy, that is there exists a morphism of complexes $\tau : F \rightarrow F'$ such that for $g \in G$ and $x \in F$, $\tau(g \cdot x) = \varphi(g) \cdot \tau(x)$ and $\epsilon' \circ \tau = \epsilon$, and $\tau$ is unique up to homotopy with these properties (see \cite{Bro} Chapter II - 6 and Chapter I - 7, Lemma 7.4).

The morphism $\tau$ then induces, for each $k \in \mathbb{Z}$, a morphism
$$T_k : L^k_{G'}(K^*) = \bigoplus_{p+q = k} Hom_{G'}(F'_p, K^q) \rightarrow \bigoplus_{p+q = k} Hom_{G}(F_p, K^q) = L^k_{G}(K^*),$$
given by the right-composition with $\tau$ : if $g \in G$, $\psi \in Hom_{G'}(F'_p, K^q)$ and $x \in F_p$, $$\varphi(g) \cdot (\psi \circ \tau(x) ) = \psi(\varphi(g) \cdot \tau(x)) = \psi \circ \tau(g \cdot x)).$$ 
Moreover, these morphisms commute with the differentials of the complexes $L^*_{G'}(K^*)$ and $L^*_{G}(K^*)$ (because $\tau : F \rightarrow F'$ is a morphism of complexes). As a consequence, we obtain a morphism of complexes of $\mathfrak{D}_+$
$$T : L^*_{G'}(K^*) \rightarrow L^*_{G}(K^*).$$

If now $K^*$ is a complex of $\mathfrak{C}^{G'}$ with filtration $J$, then it is also a complex of $\mathfrak{C}^{G}$ via $\varphi$, and we check that $T$ is a morphism of filtered complexes. Let $\psi \in Hom_{G'}(F'_p, J_{\alpha} K^q)$, then $T(\psi) = \psi \circ \tau \in Hom_{G}(F_p, J_{\alpha} K^q)$, and therefore $T(\mathcal{J}_{\alpha} L^*_{G'}(K^*)) \subset \mathcal{J}_{\alpha} L^*_{G}(K^*)$.

\end{proof}

Consider the group homomorphism $\delta : G \rightarrow G \times G~;~g \mapsto (g,g)$. According to proposition \ref{functgroup}, $\delta$ induces a morphism
$$\Lambda C_{G\times G}^*(X \times X) = L^*_{G \times G} (\mathcal{G}^{\bullet} C^*(X \times X)) \rightarrow L^*_{G} (\mathcal{G}^{\bullet} C^*(X \times X)) = \Lambda C_{G}^*(X \times X),$$
the action of $G$ on $\mathcal{G}^{\bullet} C^*(X \times X)$ being induced by the diagonal action of $G$ on $X \times X$. Since the diagonal map $\Delta$ is equivariant if we consider the diagonal action of $G$ on $X \times X$, $\Delta$ induces a morphism
$$\Lambda C_{G}^*(X \times X) \longrightarrow \Lambda C_{G}^*(X).$$

We then call cup product and denote by $\smile$ the composition
$$\Lambda C_{G}^*(X) \otimes \Lambda C_{G}^*(X) \longrightarrow \Lambda C_{G\times G}^*(X \times X) \longrightarrow \Lambda C_{G}^*(X \times X) \longrightarrow \Lambda C_{G}^*(X)$$
of $H o \, \mathfrak{C}_+$.

\begin{theo} \label{theoequivcupprod} The cup product
$$\smile \; : \Lambda C_{G}^*(X) \otimes \Lambda C_{G}^*(X) \longrightarrow \Lambda C_{G}^*(X)$$
in $H o \, \mathfrak{C}_+$
induces a morphism of spectral sequences
$$\smile \; : \bigoplus_{p+s = a, q+t = b} {}^G\!E^{p,q}_r(X) \otimes {}^G\!E^{s,t}_r(X) \longrightarrow {}^G\!E^{a,b}_r(X)$$
for $r \geq 1$, and a cup product in equivariant cohomology
$$\smile \; : H^*(X ; G) \otimes H^*(X ; G) \longrightarrow H^*(X;G)$$
which is filtered with respect to the cohomological equivariant weight filtration.
\end{theo}

\begin{rem} The cup product $\smile$ on the equivariant cohomology coincides with the cup product of \cite{VH} (III - 3 - (28)), at least for compact real algebraic $G$-varieties.
\end{rem}

The cup product on the cohomological equivariant weight spectral sequence and filtration is actually induced by the cup product on the cohomological weight spectral sequence through the Hochschild-Serre spectral sequences 
$${}^{p}\!E_2^{\alpha,\beta} = H^{\alpha}\left(G, E_1^{p, \beta}\right) \Rightarrow {}^G\!E_1^{p,\alpha + \beta} = H^{\alpha + \beta}\left(G, E_0^{p,*}\right), p \in \mathbb{Z}.$$

Indeed, consider the functor $L^*_G : \mathfrak{D}^G \rightarrow \mathfrak{D}_+$ induced by the bar resolution $B$ over $G$. If $K^* \in \mathfrak{D}^G$, we can consider the bounded filtration ${}_I\!F^{\bullet}$ on $L^*(K^*)$ which gives rise to the Hochschild-Serre spectral sequence ${}_{I}\!E_*$, and $L^*(K^*)$ can then be considered as a complex of $\mathfrak{C}_+$. We obtain a functor ${}_I\!L^* : \mathfrak{D}^G \rightarrow \mathfrak{C}_+$, which preserves quasi-isomorphisms (since ${}_{I}\!E_1^{p,q}  = Hom_G(B_p ,H^q(K^*))$) and then induces a functor  
$${}_I\!L^* : H o \, \mathfrak{D}^G \rightarrow H o \, \mathfrak{C}_+$$
If $p$ is an integer, the Hochschild-Serre spectral sequence ${}^p\!E_*$ is then the spectral sequence induced by the filtered complex ${}_I\!L^*(E_0^{p,*})$.

Now, fix two integers $p$ and $p'$. The isomorphism $\mathcal{G}^{\bullet} C^*(X) \otimes \mathcal{G}^{\bullet} C^*(X) \longrightarrow \mathcal{G}^{\bullet} C^*(X \times X)$ of $H o \, \mathfrak{C}^{G \times G}$ induces an isomorphism $E_0^{p,*}(X) \otimes E_0^{p',*}(X) \rightarrow E_0^{p+p',*}(X \times X)$ of $H o \, \mathfrak{D}^{G \times G}$. We apply the functor ${}_I\!L_{G \times G}^*$ to get a morphism 
\begin{equation} \label{morph1} {}_I\!L_{G}^*\left(E_0^{p,*}(X)\right) \otimes {}_I\!L_{G}\left(E_0^{p',*}(X)\right)  \rightarrow {}_I\!L_{G \times G}^*\left(E_0^{p,*}(X) \otimes E_0^{p',*}(X)\right) \rightarrow  {}_I\!L_{G \times G}^*\left(E_0^{p+p',*}(X \times X)\right) \tag{4.1} \end{equation} 
of $H o \, \mathfrak{C}_+$. The left-hand arrow is a natural isomorphism of $\mathfrak{C}_+$ given by the following lemma :

\begin{lem} \label{lemfI}  If $K^* \in \mathfrak{D}^G$ and $M^* \in \mathfrak{D}^{G'}$, there is a natural isomorphism
$${}_I\!L^*_G(K^*) \otimes {}_I\!L^*_{G'}(M^*) \rightarrow {}_I\!L^*_{G \times G'} (K^* \otimes M^*).$$
of filtered complexes of $\mathfrak{C}_+$.
\end{lem}

\begin{proof} The morphism is given by proposition \ref{prophomtens}, as in the proof of proposition \ref{isolprod}. We check that the filtrations coincide on the two complexes above. On the one hand, we have

\begin{eqnarray*}
{}_I\!F^l (L^*_G(K^*) \otimes L^*_{G'}(M^*))^k & = & \bigoplus_{i + j = k} \sum_{ a + b = l} {}_I\!F^a L^i_G(K^*) \otimes {}_I\!F^b L^j_{G'}(M^*) \\
& = &  \bigoplus_{i + j = k} \sum_{ a + b = l} \left(\bigoplus_{r \geq a} Hom_G(B_r , K^{i-r})\right) \otimes \left(\bigoplus_{s \geq b} Hom_G(B'_s , M^{j-s})\right) \\
& \stackrel{\sim}{\longrightarrow} &  \bigoplus_{i + j = k} \sum_{ a + b = l} \bigoplus_{r \geq a} \bigoplus_{s \geq b} Hom_{G \times G'} (B_r \otimes B'_s, K^{i-r} \otimes M^{j-s}) \\
& = & \sum_{ a + b = l} \bigoplus_{r \geq a} \bigoplus_{s \geq b} \bigoplus_i  Hom_{G \times G'} (B_r \otimes B'_s, K^{i-r} \otimes M^{k-i-s}) \\
& = &  \sum_{ a} \bigoplus_{r \geq a} \bigoplus_{s \geq l-a } \bigoplus_{\gamma} Hom_{G \times G'} (B_r \otimes B'_s, K^{\gamma} \otimes M^{k-\gamma - r-s}) \\
& = &   \bigoplus_{\gamma} \sum_{ a} \bigoplus_{r \geq a} \bigoplus_{N \geq l +r-a } Hom_{G \times G'} (B_r \otimes B'_{N-r}, K^{\gamma} \otimes M^{k-N-\gamma}), \\
\end{eqnarray*}
and on the other hand,
\begin{eqnarray*}
{}_I\!F^l L^k_{G \times G'} (K^* \otimes M^*) & = & \bigoplus_{N \geq l} Hom_{G \times G'} \left( \bigoplus_{\alpha + \beta = N} B_{\alpha} \otimes B'_{\beta} , \bigoplus_{\gamma + \delta = k - N} K^{\gamma} \otimes M^{\delta}\right)  \\
& = & \bigoplus_{N \geq l} \bigoplus_{\alpha + \beta = N} \bigoplus_{\gamma + \delta = k - N} Hom_{G \times G'} (B_{\alpha} \otimes B'_{\beta} , K^{\gamma} \otimes M^{\delta}) \\
& = & \bigoplus_{N \geq l}  \bigoplus_{\alpha + \beta = N} \bigoplus_{\gamma} Hom_{G \times G'} (B_{\alpha} \otimes B'_{\beta} , K^{\gamma} \otimes M^{k-N- \gamma}) \\
& = & \bigoplus_{\gamma} \bigoplus_{N \geq l}  \bigoplus_{r}  Hom_{G \times G'} (B_{r} \otimes B'_{N-r} , K^{\gamma} \otimes M^{k-N- \gamma}). \\
\end{eqnarray*}
Since the two sums are equal, we get the result.

\end{proof}

We then use the fact that the morphism $T : L^*_{G'}(K^*) \longrightarrow L^*_{G}(K^*)$ of $\mathfrak{D}_+$ of proposition \ref{functgroup} is compatible with the filtrations ${}_I\!F^{\bullet}$. This provides a morphism 
\begin{equation} \label{morph2} {}_I\!L_{G \times G}^*\left(E_0^{p+p',*}(X \times X)\right) \rightarrow {}_I\!L_{G}^*\left(E_0^{p+p',*}(X \times X)\right) \tag{4.2} \end{equation}
of $H o \, \mathfrak{C}_+$, induced by the group homomorphism $\delta : G \rightarrow G \times G~;~g \mapsto (g,g)$. Finally, the diagonal map $\Delta : X \rightarrow X \times X$ induces a morphism
\begin{equation} \label{morph3} {}_I\!L_{G}^*\left(E_0^{p+p',*}(X \times X)\right) \rightarrow {}_I\!L_{G}^*\left(E_0^{p+p',*}(X)\right) \tag{4.3} \end{equation}
by functoriality.

\begin{defprop} We denote by $\smile$ the composition
$${}_I\!L_{G}^*\left(E_0^{p,*}(X)\right) \otimes {}_I\!L_{G}\left(E_0^{p',*}(X)\right) \longrightarrow {}_I\!L_{G}^*\left(E_0^{p+p',*}(X)\right)$$
of the morphisms (\ref{morph1}), (\ref{morph2}) and (\ref{morph3}) of $H o \, \mathfrak{C}_+$.

It induces well-defined morphisms
$$\smile \, : \bigoplus_{\alpha + \gamma= a , \beta + \delta = b} {}^p\!E_r^{\alpha, \beta} \otimes {}^{p'}\!E_r^{\gamma, \delta} \rightarrow {}^{p+p'}\!E_r^{a,b},$$
on the spectral sequences ${}^p\!E_*$, $p \in \mathbb{Z}$, from page $r = 1$, which induce the cup products 
$$\smile \,: {}^G\!E_1^{p, *} \otimes {}^G\!E_1^{p', *} \longrightarrow {}^G\!E_1^{p +p', *},$$
and the cup products on ${}^G\!E_r$, for $r \geq 1$, and $H^*(X ; G)$, since all the cup products $\smile$ are induced by the same morphisms $u^{\vee}$, $w$, $\delta$ and $\Delta$.
\end{defprop}

Inducing the equivariant weight spectral sequence's cup product from the weight spectral sequence's cup product will allow us to carry the usual properties of the cup product on the equivariant cup product :

\begin{theo} \label{comassoccupg} Let $r \geq 1$. The cup product
$$\smile \; : \bigoplus_{p+s = a, q+t = b} {}^G\!E^{p,q}_r(X) \otimes {}^G\!E^{s,t}_r(X) \longrightarrow {}^G\!E^{a,b}_r(X)$$
is commutative. 
\end{theo}

\begin{proof} 

The commutativity of the cup product 
$$\bigoplus_{p+s = a, q+t = b} E^{p,q}_r(X) \otimes E^{s,t}_r(X) \longrightarrow E^{a,b}_r(X)$$
on the cohomological weight spectral sequence (Proposition 5.20 of \cite{LP}) is given by the commutativity of the diagram

\begin{equation} \label{diagcom} \xymatrix{
E^{p,q}_r(X) \otimes E^{s,t}_r(X)  \ar[rr]^{\Phi} \ar[d]_{\left(u^{\vee}\right)^{-1} \circ w} && E^{s,t}_r(X) \otimes E^{p,q}_r(X) \ar[d]^{\left(u^{\vee}\right)^{-1} \circ w} \\
E^{a,b}_r(X \times X) \ar[rr]^{\phi^*}  \ar[rd]_{\Delta^*} && E^{a,b}_r(X \times X) \ar[ld]^{\Delta^*} \\
& E^{a,b}_r(X)
}
\tag{4.4}
\end{equation}
where 
\begin{itemize}
	\item $p+s = a$ and $q+t = b$, 
	\item the morphism $\Phi$ is defined by $\Phi(\varphi \otimes \psi) := \psi \otimes \varphi$,
	\item the morphism $\phi$ is defined by $\phi(x,x') := (x',x)$ if $(x,x') \in X \times X$. 
\end{itemize}

The commutativity of the lower part is induced by the equality $\phi \circ \Delta = \Delta$, while the upper part is induced by the commutative diagrams
$$\xymatrix{
C_*(X) \otimes C_*(X)  \ar[rr]^{\Phi} \ar[d]_{u} && C_*(X) \otimes C_*(X) \ar[d]^{u} \\
C_*(X \times X) \ar[rr]^{\phi_*} && C_*(X \times X) \\
}
$$
and
$$\xymatrix{ 
( C_*(X))^{\vee} \otimes ( C_*(X))^{\vee} \ar[rr]^{\Phi} \ar[d]_{w} &&  C_*(X))^{\vee} \otimes ( C_*(X))^{\vee} \ar[d]^{w} \\
( C_*(X) \otimes C_*(X) )^{\vee} \ar[rr]^{\Phi^{\vee}} && ( C_*(X) \otimes C_*(X) )^{\vee}
}$$
on the filtered chain level.

We then apply group cohomology to the diagram (\ref{diagcom}) for $r=1$ and we set up the following one

\begin{equation} \label{diagcomg} \hspace{-1.5cm} \xymatrix{
H^{\mu} (G , E^{p,q}_1(X)) \otimes H^{\rho}(G, E^{s,t}_1(X) ) \ar[rr]^{\Phi} \ar[d]_{K} & & H^{\rho}(G, E^{s,t}_1(X)) \otimes H^{\mu}( G, E^{p,q}_1(X)) \ar[d]^{K} \\
H^{\alpha} (G \times G , E^{p,q}_1(X) \otimes E^{s,t}_1(X) ) \ar[r]^{\Theta} \ar[d]_{\left(u^{\vee}\right)^{-1} \circ w} & H^{\alpha} (G \times G , E^{p,q}_1(X) \otimes E^{s,t}_1(X) ) \ar[r]^{\Phi_*} \ar[d]_{\left(u^{\vee}\right)^{-1} \circ w}  & H^{\alpha}(G \times G, E^{s,t}_1(X) \otimes E^{p,q}_1(X)) \ar[d]^{\left(u^{\vee}\right)^{-1} \circ w} \\
H^{\alpha}(G \times G , E^{a,b}_1(X \times X)) \ar[d]_T \ar[r]^{\Theta} & H^{\alpha}(G \times G , E^{a,b}_1(X \times X))  \ar[r]^{\phi^*} \ar[ld]_{T} & H^{\alpha}( G \times G, E^{a,b}_1(X \times X) ) \ar[d]^T\\
H^{\alpha}(G , E^{a,b}_1(X \times X)) \ar[rr]^{\phi^*}  \ar[rd]_{\Delta^*} && H^{\alpha}( G, E^{a,b}_1(X \times X)) \ar[ld]^{\Delta^*} \\
& H^{\alpha}(G, E^{a,b}_1(X))
}
\tag{4.5}
\end{equation} 
where
\begin{itemize}
\item $\mu + \rho = \alpha$,
\item the morphism 
$$K : H^{\mu} (G , E^{p,q}_1(X)) \otimes H^{\rho}(G, E^{s,t}_1(X) ) \rightarrow H^{\alpha} (G \times G , E^{p,q}_1(X) \otimes E^{s,t}_1(X) )$$
is the composition of the K\"unneth isomorphism of cochain complexes with the natural isomorphism of proposition \ref{isolprod} (we consider $E^{p,q}_1(X)$ and $E^{s,t}_1(X)$ as cochain complexes concentrated in $0$),
\item the morphisms $\Theta$ are induced by the morphism $\theta : F_{\mu} \otimes F_{\rho} \rightarrow F_{\rho} \otimes F_{\mu} ~;~x \otimes y \mapsto y \otimes x$ if $F$ is a projective resolution over $G$,
\item $T$ is the morphism given by proposition \ref{functgroup}, induced by the augmentation-preserving $G$-chain map $\tau : F \rightarrow F \otimes F$, which is itself induced by the group homomorphism $\delta : G \rightarrow G \times G$.
\end{itemize}

We show that the diagram (\ref{diagcomg}) is commutative. 

First, we have $T \circ \Theta = T$. Indeed, $\theta \circ \tau : F \rightarrow F \otimes F$ is also an augmentation-preserving $G$-chain map and therefore, by the uniqueness up to homotopy of such a map, there is an homotopy between $\tau$ and $\theta \circ \tau$, which induces an homotopy between $(\theta \circ \tau)^*$ and $\tau^* : Hom_G(F \otimes F, M) \rightarrow Hom_G(F,M)$, for any $\mathbb{Z}[G]$-module $M$. 

The other diagrams constituting the diagram (\ref{diagcomg}) are also commutative, thanks to the commutativity of the diagram (\ref{diagcom}) and the functoriality of group cohomology. 
\\

As a consequence, the cup product 
$$\smile \, : \bigoplus_{\alpha + \gamma= a , \beta + \delta = b} {}^p\!E_2^{\alpha, \beta} \otimes {}^{p'}\!E_2^{\gamma, \delta} \rightarrow {}^{p+p'}\!E_2^{a,b},$$
is commutative and so are the induced cup products on ${}^*\!E_r$, from $r = 2$, and ${}^G\!E_r$, from $r = 1$, since all the morphisms of the diagram (\ref{diagcomg}) are induced by morphisms defined on the filtered chain level (the differentials of the spectral sequences ${}^*\!E_r$ and ${}^G\!E_r$ are therefore compatible with these morphisms). 

\end{proof}

\begin{theo} \label{comassoc2cupg} Let $r \geq 1$. The cup product
$$\smile \; : \bigoplus_{p+s = a, q+t = b} {}^G\!E^{p,q}_r(X) \otimes {}^G\!E^{s,t}_r(X) \longrightarrow {}^G\!E^{a,b}_r(X)$$
is associative. 
\end{theo}

\begin{proof}

The associativity of the cup product on the cohomological weight spectral sequence is given by the commutative diagram

\begin{equation} \label{diagassoc} \hspace{-1.5cm} \xymatrix{
& E^{p_1,q_1}_r(X) \otimes E^{p_2,q_2}_r(X) \otimes E^{p_3,q_3}_r(X) \ar[ld]_{\left(u^{\vee}\right)^{-1} \circ w \otimes id} \ar[rd]^{id \otimes \left(u^{\vee}\right)^{-1} \circ w} & \\
E^{p_1+p_2,q_1+q_2}_r(X \times X) \otimes E^{p_3,q_3}_r(X) \ar[d]_{\Delta^* \otimes id} \ar[rd]^{\left(u^{\vee}\right)^{-1} \circ w} & & E^{p_1,q_1}_r(X) \otimes E^{p_2+p_3,q_2+q_3}_r(X \times X) \ar[d]^{id \otimes \Delta^*} \ar[ld]_{\left(u^{\vee}\right)^{-1} \circ w} \\
E^{p_1+p_2,q_1+q_2}_r(X) \otimes E^{p_3,q_3}_r(X) \ar[d]_{\left(u^{\vee}\right)^{-1} \circ w} & E^{p_1+p_2+p_3,q_1+q_2+q_3}_r(X \times X \times X) \ar[ld]_{(\Delta \times id_X)^*} \ar[rd]^{(id_X \times \Delta)^*}& E^{p_1,q_1}_r(X) \otimes E^{p_2+p_3,q_2+q_3}_r(X) \ar[d]^{\left(u^{\vee}\right)^{-1} \circ w} \\
E^{p_1+p_2+p_3,q_1+q_2+q_3}_r(X \times X) \ar[rd]_{\Delta^*} & & E^{p_1+p_2+p_3,q_1+q_2+q_3}_r(X \times X) \ar[ld]^{\Delta^*} \\
& E^{p_1+p_2+p_3,q_1+q_2+q_3}_r(X) & 
}
\tag{4.6}
\end{equation}
where, if $(x,x') \in X \times X$, $\Delta \times id_X(x,x') = (x,x,x') \in X \times X \times X$.

We first make precise why the diagram (\ref{diagassoc}) is indeed commutative. The commutativity of its upper part is given by the commutativity of the following diagram, on the filtered chain level,
\begin{equation*} \hspace{-1.5cm} \xymatrix{
& (C_*(X))^{\vee} \otimes (C_*(X))^{\vee} \otimes (C_*(X))^{\vee} \ar[ld]_{w \otimes id} \ar[rd]^{id \otimes w }
& \\
(C_*(X) \otimes C_*(X))^{\vee} \otimes (C_*(X))^{\vee} \ar[rd]^{\omega} & & (C_*(X))^{\vee} \otimes (C_*(X) \otimes C_*(X))^{\vee} \ar[ld]_{\omega'} \\
(C_*(X \times X))^{\vee} \otimes (C_*(X))^{\vee}  \ar[u]^{u^{\vee} \otimes id} \ar[d]^w & (C_*(X) \otimes C_*(X) \otimes C_*(X))^{\vee} & (C_*(X))^{\vee} \otimes (C_*(X \times X))^{\vee}  \ar[u]_{id \otimes u^{\vee}} \ar[d]_w \\
(C_*(X \times X) \otimes C_*(X))^{\vee} \ar[ur]^{(u \otimes id)^{\vee}}& & (C_*(X) \otimes C_*(X \times X))^{\vee} \ar[ul]_{(id \otimes u)^{\vee}}\\
& (C_*(X \times X \times X))^{\vee} \ar[ul]^{u^{\vee}} \ar[ur]_{u^{\vee}} &  
}
\end{equation*}
where $\omega(\varphi \otimes \psi)(c_1 \otimes c_2 \otimes c_3) = \varphi(c_1 \otimes c_2) \cdot \psi(c_3)$.

The lower part of (\ref{diagassoc}) is commutative by functoriality, since $(\Delta \times id_X) \circ \Delta = (id_X \times \Delta) \circ \Delta$. The commutativity of the left and right parts of (\ref{diagassoc}) is then given by the following property of the cohomological weight spectral sequence :

\begin{lem} \label{functprodcross} Let $f : Y \rightarrow Y'$ and $h : Z \rightarrow Z'$ be two morphisms of $\mathbf{Sch}_c(\mathbb{R})$. Then the following diagram of spectral sequences is commutative, from $r = 1$,
\begin{equation*} \xymatrix{
E_r^{p,q}(Y') \otimes E_r^{s,t}(Z') \ar[r]^{f^* \otimes h^*} \ar[d]^{(u^{\vee})^{-1} \circ w} & E_r^{p,q}(Y) \otimes E_r^{s,t}(Z) \ar[d]^{(u^{\vee})^{-1} \circ w} \\  
E_r^{p+s,q+t}(Y' \times Z') \ar[r]^{(f \times h)^*} & E_r^{p+s,q+t}(Y \times Z)
}
\end{equation*}

It is induced by the commutative diagrams 
\begin{equation*} \xymatrix{
C_*(Y) \otimes C_*(Z) \ar[r]^{f_* \otimes h_*} \ar[d]^{u} & C_*(Y') \otimes C_*(Z') \ar[d]^{u} \\  
C_*(Y \times Z) \ar[r]^{(f \times h)_*} & C_*(Y' \times Z')
}
\end{equation*}
and
\begin{equation*} \xymatrix{
(C_*(Y))^{\vee} \otimes (C_*(Z))^{\vee} \ar[d]^{w} & (C_*(Y'))^{\vee} \otimes (C_*(Z'))^{\vee}  \ar[l]_{f^* \otimes h^*} \ar[d]^{w} \\  
(C_*(Y) \otimes C_*(Z))^{\vee}  & (C_*(Y') \otimes C_*(Z'))^{\vee} \ar[l]_{(f_* \otimes h_*)^{\vee}}
}
\end{equation*}
on the filtered chain level.
\end{lem}

Now, we want to show the associativity of the cup product 
$$\smile \, : \bigoplus_{\alpha + \gamma= a , \beta + \delta = b} {}^p\!E_2^{\alpha, \beta} \otimes {}^{p'}\!E_2^{\gamma, \delta} \rightarrow {}^{p+p'}\!E_2^{a,b},$$
that is the commutativity of the following diagram (\ref{diagassocg0}). For the sake of readability, we denote 
\begin{itemize}
	\item $H^*_G(\cdot) := H^*(G, \cdot)$,
	\item $E_1 := E_1(X)$,
	\item $\rho_{i,j} := \rho_{i} + \rho_j$ and $\rho := \rho_1 + \rho_2 + \rho_3$ if $\rho = p,q$ or $\mu$.
\end{itemize}

\begin{equation}  \label{diagassocg0} {\tiny \hspace{-1cm} \xymatrix{
& H_G^{\mu_1}\left(E_1^{p_1,q_1}\right) \otimes H_G^{\mu_2}\left(E_1^{p_2,q_2}\right) \otimes H_G^{\mu_3}\left(E_1^{p_3,q_3}\right) \ar[dr]^{id \otimes K} \ar[ld]_{K \otimes id} & \\
H_{G \times G}^{\mu_{1,2}}\left(E_1^{p_1,q_1} \otimes E_1^{p_2,q_2}\right) \otimes H_G^{\mu_3}\left(E_1^{p_3,q_3}\right) \ar[d]_{\left(u^{\vee}\right)^{-1} \circ w \otimes id} & & H_G^{\mu_1}\left(E_1^{p_1,q_1}\right) \otimes H_{G \times G}^{\mu_{2,3}}\left(E_1^{p_2,q_2} \otimes E_1^{p_3,q_3}\right) \ar[d]^{id \otimes \left(u^{\vee}\right)^{-1} \circ w} \\
H_{G \times G}^{\mu_{1,2}}\left(E^{p_{1,2},q_{1,2}}_1(X \times X)\right) \otimes H_G^{\mu_3}\left(E_1^{p_3,q_3}\right) \ar[d]_{T \otimes id} & & H_G^{\mu_1}\left(E_1^{p_1,q_1}\right) \otimes H_{G \times G}^{\mu_{2,3}}\left(E^{p_{2,3},q_{2,3}}_1(X \times X)\right) \ar[d]^{id \otimes T} \\
H_G^{\mu_{1,2}}\left(E^{p_{1,2},q_{1,2}}_1(X \times X)\right) \otimes H_G^{\mu_3}\left(E_1^{p_3,q_3}\right) \ar[d]_{\Delta^* \otimes id} & & H_G^{\mu_1}\left(E_1^{p_1,q_1}\right) \otimes H_G^{\mu_{2,3}}\left(E^{p_{2,3},q_{2,3}}_1(X \times X)\right) \ar[d]^{id \otimes \Delta^*} \\
H_G^{\mu_{1,2}}\left(E^{p_{1,2},q_{1,2}}_1\right) \otimes H_G^{\mu_3}\left(E_1^{p_3,q_3}\right) \ar[d]_K & & H_G^{\mu_1}(E_1^{p_1,q_1}) \otimes H_G^{\mu_{2,3}}\left(E^{p_{2,3},q_{2,3}}_1\right) \ar[d]^K \\
H_{G \times G}^{\mu}\left(E^{p_{1,2},q_{1,2}}_1 \otimes E_1^{p_3,q_3}\right) \ar[d]_{\left(u^{\vee}\right)^{-1} \circ w} & & H_{G \times G}^{\mu}\left(E^{p_{1},q_{1}}_1 \otimes E_1^{p_{2,3},q_{2,3}}\right) \ar[d]^{\left(u^{\vee}\right)^{-1} \circ w} \\
H_{G \times G}^{\mu}\left(E_1^{p,q}(X \times X)\right) \ar[d]_T & & H_{G \times G}^{\mu}\left(E_1^{p,q}(X \times X)\right) \ar[d]^T \\
H_G^{\mu}\left(E_1^{p,q}(X \times X)\right) \ar[rd]^{\Delta^*} & & H_G^{\mu}\left(E_1^{p,q}(X \times X)\right) \ar[ld]_{\Delta^*} \\
& H_G^{\mu}\left(E_1^{p,q}\right) &
}
}
\tag{4.7}
\end{equation}

In order to prove that the diagram (\ref{diagassocg0}) is commutative, we fill it with commutative diagrams in the following way. We show up below the left part of the obtained diagram, the right part being symmetric :

\begin{landscape} 
\begin{equation} \label{diagassocg} {\tiny  \xymatrix{
&&& H_G^{\mu_1}\left(E_1^{p_1,q_1}\right) \otimes H_G^{\mu_2}\left(E_1^{p_2,q_2}\right) \otimes H_G^{\mu_3}\left(E_1^{p_3,q_3}\right) \ar[dl]_{K \otimes id} \\
& H_{G \times G}^{\mu_{1,2}}\left(E^{p_{1,2},q_{1,2}}_1(X \times X)\right) \otimes H_G^{\mu_3}\left(E_1^{p_3,q_3}\right) \ar[d]_{T \otimes id} & H_{G \times G}^{\mu_{1,2}}\left(E_1^{p_1,q_1} \otimes E_1^{p_2,q_2}\right) \otimes H_G^{\mu_3}\left(E_1^{p_3,q_3}\right) \ar[l]_{\raise0.3cm\hbox{$(u^{\vee})^{-1} \circ w \otimes id$}} \ar[r]^K \ar[d]^{T \otimes id} & H_{G \times G \times G}^{\mu} \left(E_1^{p_1,q_1} \otimes E_1^{p_2,q_2} \otimes E_1^{p_3,q_3}\right) \ar[dd]^{T_0} \ar[ddl]^{T_1} \\
H_G^{\mu_{1,2}}\left(E^{p_{1,2},q_{1,2}}_1\right) \otimes H_G^{\mu_3}\left(E_1^{p_3,q_3}\right) \ar[d]_K & H_G^{\mu_{1,2}}\left(E^{p_{1,2},q_{1,2}}_1(X \times X)\right) \otimes H_G^{\mu_3}\left(E_1^{p_3,q_3}\right) \ar[l]_{\raise0.3cm\hbox{$\Delta^* \otimes id$}} \ar[d]_K & H_{G}^{\mu_{1,2}}\left(E_1^{p_1,q_1} \otimes E_1^{p_2,q_2}\right) \otimes H_G^{\mu_3}\left(E_1^{p_3,q_3}\right) \ar[l]_{\raise0.3cm\hbox{$(u^{\vee})^{-1} \circ w \otimes id$}} \ar[d]^K & \\
H_{G \times G}^{\mu}\left(E^{p_{1,2},q_{1,2}}_1 \otimes E_1^{p_3,q_3}\right) \ar[dd]_{(u^{\vee})^{-1} \circ w} & H_{G \times G}^{\mu}\left(E^{p_{1,2},q_{1,2}}_1(X \times X) \otimes E_1^{p_3,q_3}\right) \ar[l]_{\Delta^* \otimes id} \ar[dd]_{(u^{\vee})^{-1} \circ w} \ar[rd]_T & H_{G \times G}^{\mu}\left(E_1^{p_1,q_1} \otimes E_1^{p_2,q_2} \otimes E_1^{p_3,q_3}\right) \ar[l]_{(u^{\vee})^{-1} \circ w \otimes id} \ar[r]^T & H_{G}^{\mu}\left(E_1^{p_1,q_1} \otimes E_1^{p_2,q_2} \otimes E_1^{p_3,q_3}\right) \ar[ld]^{(u^{\vee})^{-1} \circ w \otimes id} \\
& & H_G^{\mu}\left(E^{p_{1,2},q_{1,2}}_1(X \times X) \otimes E_1^{p_3,q_3}\right) \ar[rd]_{(u^{\vee})^{-1} \circ w} & \\
H_{G \times G}^{\mu}\left(E_1^{p,q}(X \times X)\right) \ar[d]_T  & H^{\mu}_{G \times G}\left(E_1^{p,q}(X \times X \times X)\right) \ar[l]_{\Delta_1^*} \ar[rr]^T & & H_G^{\mu}\left(E_1^{p,q}(X \times X \times X)\right) \ar[d]^{\Delta_0^*} \ar[llld]_{\Delta_1^*} \\
H_G^{\mu}\left(E_1^{p,q}(X \times X)\right) \ar[rrr]_{\Delta^*} & & & H_G^{\mu}\left(E_1^{p,q}\right)
}}
\tag{4.8}
\end{equation}
\end{landscape}
~~\\
In the diagram (\ref{diagassocg}) above and in its symmetric,
\begin{itemize}
	\item for $(x,y) \in X \times X$, $\Delta_1(x,y) := (x,x,y) \in X \times X \times X$ and $\Delta_2(x,y) := (x,y,y) \in X \times X \times X$,
	\item for $x \in X$, $\Delta_0(x) := (x,x,x)$,
	\item $T_0$ is the morphism given by proposition \ref{functgroup}, induced by the augmentation-preserving $G$-chain map $\tau_0 : F \rightarrow F \otimes F \otimes F$,  which is itself induced by the group homomorphism $G \rightarrow G \times G \times G~;~g \mapsto (g,g,g)$,
	\item $T_1$ is the morphism given by proposition \ref{functgroup}, induced by the augmentation-preserving $G \times G$-chain map $\tau_1 : F \otimes F \rightarrow (F \otimes F) \otimes F$, which is itself induced by the group homomorphism $G \times G \rightarrow (G \times G) \times G~;~(g,g') \mapsto (g,g,g')$,
	\item $T_2$ is the morphism given by proposition \ref{functgroup}, induced by the augmentation-preserving $G \times G$-chain map $\tau_2 : F \otimes F \rightarrow F \otimes (F \otimes F)$, which is itself induced by the group homomorphism $G \times G \rightarrow G \times (G \times G)~;~(g,g') \mapsto (g,g',g')$. 
\end{itemize}
Notice that, by uniqueness up to homotopy of the augmentation-preserving $G \times G$-chain map $F \otimes F \rightarrow (F \otimes F) \otimes F$ induced by the group homomorphism $G \times G \rightarrow (G \times G) \times G~;~(g,g') \mapsto (g,g,g')$, $\tau_1$ is homotopic to $\tau \otimes id$. In the same manner, $\tau_2$ is homotopic to $id \otimes \tau$. Furthermore, $\tau_0$ is homotopic to $\tau_1 \circ \tau$ and $\tau_2 \circ \tau$, by uniqueness up to homotopy of the augmentation-preserving $G$-chain map $F \rightarrow F \otimes F \otimes F$ induced by the group homomorphism $G \rightarrow G \times G \times G~;~g \mapsto (g,g,g)$.
\\

Finally, the diagrams
{\tiny
\begin{equation*}  \hspace{-1cm} \xymatrix{
& H_G^{\mu_1}\left(E_1^{p_1,q_1}\right) \otimes H_G^{\mu_2}\left(E_1^{p_2,q_2}\right) \otimes H_G^{\mu_3}\left(E_1^{p_3,q_3}\right) \ar[dl]_{K \otimes id} \ar[dr]^{id \otimes K} & \\
H_{G \times G}^{\mu_{1,2}}\left(E_1^{p_1,q_1} \otimes E_1^{p_2,q_2}\right) \otimes H_G^{\mu_3}\left(E_1^{p_3,q_3}\right) \ar[r]^K & H_{G \times G \times G}^{\mu} \left(E_1^{p_1,q_1} \otimes E_1^{p_2,q_2} \otimes E_1^{p_3,q_3}\right) & H_G^{\mu_1}\left(E_1^{p_1,q_1}\right) \otimes H_{G \times G}^{\mu_{2,3}}\left(E_1^{p_2,q_2} \otimes E_1^{p_3,q_3}\right) \ar[l]_K
}
\end{equation*}
}
and
{\tiny
\begin{equation*} \hspace{-1cm} \xymatrix{
& H_{G}^{\mu}\left(E_1^{p_1,q_1} \otimes E_1^{p_2,q_2} \otimes E_1^{p_3,q_3}\right) \ar[ld]^{(u^{\vee})^{-1} \circ w \otimes id} \ar[rd]_{id \otimes (u^{\vee})^{-1} \circ w} & \\
H_G^{\mu}\left(E^{p_{1,2},q_{1,2}}_1(X \times X) \otimes E_1^{p_3,q_3}\right) \ar[rd]_{(u^{\vee})^{-1} \circ w} & & H_G^{\mu}\left(E_1^{p_1,q_1} \otimes E^{p_{2,3},q_{2,3}}_1(X \times X) \right) \ar[ld]^{(u^{\vee})^{-1} \circ w} \\
& H_G^{\mu}\left(E_1^{p,q}(X \times X \times X)\right) &
}
\end{equation*}
}
~~\\
connecting the left and right parts of diagram (\ref{diagassocg0}) are also commutative : the first one by associativity of the K\"unneth isomorphism and the second one by functoriality of group cohomology applied to the topmost part of the diagram (\ref{diagassoc})).
\\

The associativity of the cup product 
$$\smile \, : \bigoplus_{\alpha + \gamma= a , \beta + \delta = b} {}^p\!E_2^{\alpha, \beta} \otimes {}^{p'}\!E_2^{\gamma, \delta} \rightarrow {}^{p+p'}\!E_2^{a,b}$$
then induces the associativity of the cup products on ${}^*\!E_r$, $r \geq 2$, and on ${}^G\!E_r$, $r \geq 1$, by the same arguments as for the commutativity. 
\end{proof}

We conclude this section by showing that the cup product is functorial in the following meaning :

\begin{theo} \label{functcupg} Let $Y$ be a real algebraic $G$-variety and $f : X \rightarrow Y$ an equivariant morphism. For all $r \geq 1$ and all $p,q,s,t \in \mathbb{Z}$, the diagram
\begin{equation*} \xymatrix{
{}^G\!E_r^{p,q}(Y) \otimes {}^G\!E_r^{s,t}(Y) \ar[r]^{f^* \otimes f^*} \ar[d]_{\smile} & {}^G\!E_r^{p,q}(X) \otimes {}^G\!E_r^{s,t}(X) \ar[d]^{\smile} \\  
{}^G\!E_r^{p+s,q+t}(Y) \ar[r]^{f^*} & {}^G\!E_r^{p+s,q+t}(X)
}
\end{equation*}
is commutative.
\end{theo}

\begin{rem} It is just the usual formula
$$f^*(c \smile c') = f^*(c) \smile f^*(c')$$ 
of cup product.
\end{rem}

\begin{proof}
The functoriality of the cup on the cohomological weight spectral sequence is given by the commutativity of the diagram 

\begin{equation*} \xymatrix{
E_r^{p,q}(Y) \otimes E_r^{s,t}(Y) \ar[r]^{f^* \otimes f^*} \ar[d]_{(u^{\vee})^{-1} \circ w} & E_r^{p,q}(X) \otimes E_r^{s,t}(X) \ar[d]^{(u^{\vee})^{-1} \circ w} \\  
E_r^{p+s,q+t}(Y \times Y) \ar[r]^{(f \times f)^*} \ar[d]_{\Delta^*} & E_r^{p+s,q+t}(X \times X) \ar[d]^{\Delta^*} \\
E_r^{p+s,q+t}(Y) \ar[r]^{f^*}  & E_r^{p+s,q+t}(X)
}
\end{equation*}
for $r \geq 1$ (it is commutative thanks to lemma \ref{functprodcross} and the equality $\Delta \circ f = (f \times f) \circ \Delta$).

The commutative diagram establishing the functoriality of the cup product on ${}^*\!E_2$ is then  

\begin{equation*} \hspace{-1.5cm} \xymatrix{
H^{\mu} (G , E^{p,q}_1(Y)) \otimes H^{\rho}(G, E^{s,t}_1(Y) ) \ar[r]^{f^* \otimes f^*} \ar[d]_{K} & H^{\mu}(G, E^{p,q}_1(X)) \otimes H^{\rho}(G, E^{s,t}_1(X)) \ar[d]^{K} \\
H^{\alpha} (G \times G , E^{p,q}_1(Y) \otimes E^{s,t}_1(Y) ) \ar[d]_{\left(u^{\vee}\right)^{-1} \circ w} \ar[r]^{f^* \otimes f^*} & H^{\alpha}(G \times G, E^{p,q}_1(X) \otimes E^{s,t}_1(X)) \ar[d]^{\left(u^{\vee}\right)^{-1} \circ w} \\
H^{\alpha}(G \times G , E^{a,b}_1(Y \times Y)) \ar[d]_T \ar[r]^{(f \times f)^*} & H^{\alpha}( G \times G, E^{a,b}_1(X \times X) ) \ar[d]^T\\
H^{\alpha}(G , E^{a,b}_1(Y \times Y)) \ar[r]^{(f \times f)^*}  \ar[d]_{\Delta^*} & H^{\alpha}( G, E^{a,b}_1(X \times X)) \ar[d]^{\Delta^*} \\
H^{\alpha}(G, E^{a,b}_1(Y)) \ar[r]^{f^*} & H^{\alpha}(G, E^{a,b}_1(X))
}
\end{equation*}  
(with $\alpha := \mu + \rho$, $a := p + s$ and $b := q + t$).

\end{proof}

\begin{rem} The commutativity, the associativity and the functoriality of the cup product on ${}^G\!E_r$, $r \geq \,1$, induces the commutativity, the associativity and the functoriality of the cup product on the equivariant cohomology $H^*(X ; G)$.
\end{rem}

\subsection{Equivariant cap product} \label{subsectcap}

In this last paragraph, we define a cap product on the Hochschild-Serre spectral sequences ${}^*\!E_r$ and ${}^*\!E^r$, $r \geq 1$, of $X$, induced by the cap product on its weight spectral sequences (see \cite{LP} Propositions 5.20). This equivariant cap product will induce one on its equivariant weight spectral sequences, which itself induces a cap product on the equivariant cohomology and homology of $X$, showing in particular that this last cap product is filtered with respect to the equivariant weight filtrations. Similarly to what we did for the equivariant cup product, we carry the usual properties of cap product from the weight spectral sequences to the equivariant weight spectral sequences.
\\

We consider the definition of the cap product given in Remark 5.24 of \cite{LP} : if $h$ denotes the filtered morphism of $\mathcal{C}$
$$\begin{array}{ccc}
\mathcal{G}^{\bullet} C^*(X) \otimes (\mathcal{G}_{\bullet} C_*(X) \otimes \mathcal{G}_{\bullet} C_*(X)) & \rightarrow & \mathcal{G}_{\bullet} C_*(X) \\
\varphi \otimes (a \otimes b) & \mapsto & \varphi(a) \cdot b
\end{array}$$
(see Definition 5.21 of \cite{LP} for the definition of the tensor product of a complex of $\mathfrak{C}$, resp. $\mathfrak{C}_+$, and a complex of $\mathcal{C}$, resp. $\mathcal{C}_-$), then the cap product of $X$ can be defined in the localized category $H o \, \mathfrak{C}$ as the composition

$$\hspace{-1cm} \frown \, : \mathcal{G}^{\bullet} C^*(X) \otimes \mathcal{G}_{\bullet} C_*(X) \xrightarrow{id \otimes \Delta_*}  \mathcal{G}^{\bullet} C^*(X) \otimes \mathcal{G}_{\bullet} C_*(X \times X) \xrightarrow{ id \otimes u^{-1} } \mathcal{G}^{\bullet} C^*(X) \otimes (\mathcal{G}_{\bullet} C_*(X) \otimes \mathcal{G}_{\bullet} C_*(X)) \xrightarrow{h} \mathcal{G}_{\bullet} C_*(X).$$

If we take into account the action of $G$ on $X$, the leftmost morphism $id \otimes \Delta_*$ induces a morphism
$$\Lambda C^*_G(X) \otimes \Lambda C_*^G(X) \rightarrow \Lambda C^*_G(X) \otimes \Lambda C_*^G(X \times X).$$
Furthermore, if $K^*$ is cochain complex of $\mathfrak{C}^G$ and $M_*$ is a chain complex of $\mathcal{C}^{G'}$, then the tensor product of $K^* \otimes M_*$ can be naturally considered as an element of $\mathcal{C}^{G \times G'}$ and we have a natural isomorphism of $\mathcal{C_-}$
$$L_G^{B}(K^*) \otimes L^{G'}_{B'}(M_*) \rightarrow L_{B \otimes B'}^{G \times G'}(K^* \otimes M_*),$$
so that the group homomorphism $\delta : G \rightarrow G \times G~;~g \mapsto (g,g)$ induces (by the homological version of proposition \ref{functgroup}) a morphism
$$\Lambda C^*_G(X) \otimes \Lambda C_*^G(X \times X) \xrightarrow{\cong} L_*^{G \times G} (\mathcal{G}^{\bullet} C^*(X) \otimes \mathcal{G}_{\bullet} C_*(X \times X)) \rightarrow L_*^G(\mathcal{G}^{\bullet} C^*(X) \otimes \mathcal{G}_{\bullet} C_*(X \times X)).$$
Finally, since the above morphism $id \otimes u^{-1}$ and $h$ are equivariant with respect to the diagonal actions of $G$, they induce morphisms
$$L_*^G(\mathcal{G}^{\bullet} C^*(X) \otimes \mathcal{G}_{\bullet} C_*(X \times X)) \rightarrow L_*^G(\mathcal{G}^{\bullet} C^*(X) \otimes (\mathcal{G}_{\bullet} C_*(X) \otimes \mathcal{G}_{\bullet} C_*(X))) \rightarrow \Lambda C^G_*(X),$$
and we call cap product, and denote by $\frown$, the composition
$$\Lambda C^*_G(X) \otimes \Lambda C_*^G(X) \rightarrow \Lambda C^*_G(X) \otimes \Lambda C_*^G(X \times X) \rightarrow L_*^G(\mathcal{G}^{\bullet} C^*(X) \otimes \mathcal{G}_{\bullet} C_*(X \times X)) \rightarrow \Lambda C^G_*(X).$$

\begin{theo} \label{theoequivcapprod} The cap product
$$\frown \; : \Lambda C_{G}^*(X) \otimes \Lambda C^{G}_*(X) \longrightarrow \Lambda C^{G}_*(X)$$
in $H o \, \mathcal{C}_-$
induces a morphism of spectral sequences
$$\frown \; : \bigoplus_{s-p = a, t-q = b} {}^G\!E^{p,q}_r(X) \otimes {}^G\!E_{s,t}^r(X) \longrightarrow {}^G\!E_{a,b}^r(X)$$
for $r \geq 1$, and a cap product in equivariant cohomology and homology
$$\frown \; : H^*(X ; G) \otimes H_*(X ; G) \longrightarrow H_*(X;G)$$
which is filtered with respect to the cohomological and homological equivariant weight filtrations.
\end{theo}

\begin{rem} We do not know if the above cap product on equivariant cohomology and homology is the same as the one of \cite{VH} (III - 3 - (31)). 
\end{rem}

We now induce the cap product from the cohomological and homological weight spectral sequences through the Hochschild-Serre spectral sequences ${}^*\!E_r$ and ${}^*\!E^r$, $r \geq 1$.

First, we define the homological analog ${}_I\!L_* : \mathcal{D}^G \rightarrow \mathcal{C}_-$ of the functor ${}_I\!L^*$, and extend it into a functor $H o \, \mathcal{D}^G \rightarrow H o \, \mathcal{C}_-$ : it associates to a complex $K_*$ of $\mathcal{D}^G$ the complex $L_*(K_*)$ equipped with the filtration ${}_I\!F_{\bullet}$ which induces the Hochschild-Serre spectral sequence ${}_{I}\!E_1^{p,q}  = Hom_G(F_{-p} ,H^q(K_*)) \Rightarrow H_{p+q} (L_*(K_*)) = H_{p+q}(G, K_*)$.

We then apply this functor ${}_I\!L_*$ to the equivariant morphism $\Delta_*$, which induces a morphism
\begin{equation*}  {}_I\!L_{G}^*\left(E_0^{p,*}(X)\right) \otimes {}_I\!L^{G}_*\left(E^0_{p',*}(X)\right)  \rightarrow  {}_I\!L_{G}^*\left(E_0^{p,*}(X)\right) \otimes {}_I\!L^{G}_*\left(E^0_{p',*}(X \times X)\right), \end{equation*}
and we use a natural isomorphism of $\mathcal{C}_-$ similar to the one in lemma \ref{lemfI} to induce, via $\delta$, a second morphism
$$\hspace{-1.5cm} {}_I\!L_{G}^*\left(E_0^{p,*}(X)\right) \otimes {}_I\!L^{G}_*\left(E^0_{p',*}(X \times X)\right) \xrightarrow{\cong} {}_I\!L^{G \times G}_*\left(E_0^{p,*}(X) \otimes E^0_{p',*}(X \times X) \right) \rightarrow {}_I\!L^{G}_*\left(E_0^{p,*}(X) \otimes E^0_{p',*}(X \times X) \right).$$
Finally, $h \circ (id \otimes u^{-1})$ induces a third morphism
$${}_I\!L^{G}_*\left(E_0^{p,*}(X) \otimes E^0_{p',*}(X \times X) \right) \rightarrow {}_I\!L^{G}_*\left(E^0_{p'-p,*}(X)\right),$$
and we obtain a cap product
$$\frown \, : {}_I\!L_{G}^*\left(E_0^{p,*}(X)\right) \otimes {}_I\!L^{G}_*\left(E^0_{p',*}(X)\right) \rightarrow {}_I\!L^{G}_*\left(E^0_{p'-p,*}(X)\right)$$
in $H o \, \mathcal{C}_-$ which induces well-defined cap products 
$$\frown \, : \bigoplus_{\gamma - \alpha = a , \delta - \beta = b} {}^p\!E_r^{\alpha, \beta} \otimes {}^{p'}\!E^r_{\gamma, \delta} \rightarrow {}^{p'- p}\!E^r_{a,b},$$
on the spectral sequences ${}^*\!E_*$ and ${}^*\!E^*$, from page one, which induce the cap products 
$${}^G\!E_1^{p, *} \otimes {}^G\!E^1_{p', *} \rightarrow {}^G\!E^1_{p' - p, *}$$
and the cap products on ${}^G\!E_r$ and ${}^G\!E^r$, for $r \geq 1$ and the cap product on $H^*(X ; G)$ and $H_*(X ; G)$ (all the cap products are induced by the morphisms $\Delta$, $\delta$, $u$ and $h$).
\\

We will thereafter induce the usual properties of cap product from the weight spectral sequences onto the equivariant ones. Notice that, since there is no duality between the cohomological and homological equivariant spectral sequences (remark \ref{remnodualg}), we have no hope to obtain a formula of the type $\psi(\varphi \frown c) = (\psi \smile \varphi)(c)$. Nevertheless, such a property is verified for the weight spectral sequences and we will use it to establish the other formulae of the cap product :

\begin{prop} \label{propwssformun} Let $r \geq 1$ and let $p,q,s,t \in \mathbb{Z}$. Now, let $\varphi \in E_r^{p,q}$ and $c \in E^r_{s,t}$. Then $\varphi \frown c$ is the unique element of $E^r_{s-p,t-q}$ such that for all $\psi \in E_r^{s-p,t-q} = (E^r_{s-p,t-q})^{\vee}$, 
$$\psi(\varphi \frown c) = (\psi \smile \varphi)(c).$$
\end{prop}

\begin{proof} If $\displaystyle{u^{-1}(\Delta_*(c)) = \bigoplus_{\alpha + \gamma = s, \beta + \delta = t} a_{\alpha, \beta} \otimes b_{\gamma, \delta}} \in \bigoplus_{\alpha + \gamma = s, \beta + \delta = t} E^r_{\alpha, \beta} \otimes E^r_{\gamma, \delta}$ then 
$$\varphi \frown c = h\left( \varphi \otimes \left(\bigoplus_{\alpha + \gamma = s, \beta + \delta = t} a_{\alpha, \beta} \otimes b_{\gamma, \delta}\right)\right) = \bigoplus_{\alpha + \gamma = s, \beta + \delta = t} \varphi(a_{\alpha, \beta}) \cdot b_{\gamma, \delta}.$$
Therefore, if $\psi \in E_r^{s-p,t-q} = (E^r_{s-p,t-q})^{\vee}$,
$$\psi(\varphi \frown c) = \bigoplus_{\alpha + \gamma = s, \beta + \delta = t} \varphi(a_{\alpha, \beta}) \cdot \psi(b_{\gamma, \delta}).$$ 

On the other hand, since $\left( u^{\vee} \right)^{-1} = \left( u^{-1} \right)^{\vee}$, we have
\begin{eqnarray*}
(\psi \smile \varphi)(c) & = & (\varphi \smile \psi )(c) \\
				&  = & \Delta^* \circ \left( u^{-1} \right)^{\vee} \circ w (\varphi \otimes \psi)(c) \\
			   & = & w (\varphi \otimes \psi) \circ u^{-1} \circ \Delta_* (c) \\
			   & = & \bigoplus_{\alpha + \gamma = s, \beta + \delta = t} \varphi(a_{\alpha, \beta}) \cdot \psi(b_{\gamma, \delta}) \\
			   & = & \psi(\varphi \frown c)
\end{eqnarray*}

The uniqueness comes from the fact that the linear space $E^r_{s-p,t-q}$ is finite-dimensional.
\end{proof}

\begin{theo} \label{prop1equivcap}
Let $r \geq 1$ and let $p,q,p',q',s,t \in \mathbb{Z}$. If $\varphi \in {}^G\!E_r^{p,q}$, $c \in {}^G\!E^r_{s,t}$ and $\psi \in {}^G\!E_r^{p',q'}$, then
\begin{equation} \label{eqform1equivcap}
(\psi \smile \varphi) \frown c = \psi \frown (\varphi \frown c).
\tag{4.9}
\end{equation}
\end{theo}

\begin{proof} We first show formula (\ref{eqform1equivcap}) on the weight spectral sequences, using proposition \ref{propwssformun} : if $\varphi_1 \in E_r^{p_1,q_1}$ and $\varphi_2 \in E_r^{p_2,q_2}$, then, for all $\psi \in E_r^{s-(p_1+p_2),t-(q_1+q_2)}$, we have
\begin{eqnarray*}
\psi(\varphi_1 \frown (\varphi_2 \frown c)) & = & (\psi \smile \varphi_1)  (\varphi_2 \frown c) \\
							        & = & ((\psi \smile \varphi_1) \smile \varphi_2)(c) \\
							        & = & (\psi \smile (\varphi_1 \smile \varphi_2))(c) \\
							        & = & \psi((\varphi_1 \smile \varphi_2) \frown c).
\end{eqnarray*}

We then express formula (\ref{eqform1equivcap}) on the non-equivariant weight spectral sequences with the following commutative diagram
\begin{equation*} \xymatrix{ 
& E^{p',q'}_r \otimes E^{p,q}_r \otimes E^r_{s,t} \ar[ld]_{\left(u^{\vee}\right)^{-1} \circ w \otimes id~~} \ar[rd]^{~id \otimes id \otimes \Delta_*} & \\
E^{p+p',q+q'}_r(X \times X) \otimes E^r_{s,t}  \ar[d]_{\Delta^* \otimes id} & & E^{p',q'}_r \otimes E^{p,q}_r \otimes E^r_{s,t}(X \times X) \ar[d]^{id \otimes h \circ (id \otimes u^{-1})} \\
E^{p+p',q+q'}_r \otimes E^r_{s,t} \ar[d]_{id \otimes \Delta_*} & & E^{p',q'}_r \otimes E^r_{s-p,t-q} \ar[d]^{id \otimes \Delta_*} \\
E^{p+p',q+q'}_r \otimes E^r_{s,t}(X \times X) \ar[rd]_{h \circ (id \otimes u^{-1})~~} & & E^{p',q'}_r \otimes E^r_{s-p,t-q}(X \times X) \ar[ld]^{~~h \circ (id \otimes u^{-1})} \\
& E^r_{s-(p+p'), t-(q+q')} & 
}
\end{equation*}

The next step consists in showing that formula (\ref{eqform1equivcap}) is true on the Hochschild-Serre spectral sequences ${}^*\!E_2$ and ${}^*\!E^2$, that is we are going to prove that the following diagram is commutative~: 
 \begin{equation} \label{equivcapformcup} {\tiny \hspace{-1.5cm} \xymatrix{
& H_G^{\mu'}\left(E^{p',q'}_1\right) \otimes H_G^{\mu}\left(E^{p,q}_1\right) \otimes H_G^{-\rho}\left(E^1_{s,t}\right) \ar[ld]_{K \otimes id} \ar[rd]^{id \otimes id \otimes \Delta_*} & \\
H_{G \times G}^{\mu'+\mu}\left(E^{p',q'}_1 \otimes E^{p,q}_1\right) \otimes H_G^{-\rho}\left(E^1_{s,t}\right) \ar[d]_{\left(u^{\vee}\right)^{-1} \circ w \otimes id} & & H_G^{\mu'}\left(E^{p',q'}_1\right) \otimes H_G^{\mu}\left(E^{p,q}_1\right) \otimes H_G^{-\rho}\left(E^1_{s,t}(X \times X)\right) \ar[d]^{id \otimes K} \\
H_{G \times G}^{\mu'+\mu}\left(E^{p'+p,q'+q}_1 (X \times X)\right) \otimes H_G^{-\rho}\left(E^1_{s,t}\right) \ar[d]_{T \otimes id} & &  H_G^{\mu'}\left(E^{p',q'}_1\right) \otimes H_{G \times G}^{- \rho + \mu} \left(E^{p,q}_1 \otimes E^1_{s,t}(X \times X)\right) \ar[d]^{id \otimes T} \\
H_{G}^{\mu'+\mu}\left(E^{p'+p,q'+q}_1 (X \times X)\right) \otimes H_G^{-\rho}\left(E^1_{s,t}\right) \ar[d]_{\Delta^* \otimes id} & &  H_G^{\mu'}\left(E^{p',q'}_1\right) \otimes H_{G}^{-\rho + \mu} \left(E^{p,q}_1 \otimes E^1_{s,t}(X \times X)\right) \ar[d]^{id \otimes h \circ (id \otimes u^{-1})} \\
H_{G}^{\mu'+\mu}\left(E^{p'+p,q'+q}_1\right) \otimes H_G^{-\rho}\left(E^1_{s,t}\right) \ar[d]_{id \otimes \Delta_*} & &  H_G^{\mu'}\left(E^{p',q'}_1\right) \otimes H_{G}^{-\rho + \mu} \left(E_{s-p,t-q}^1\right) \ar[d]^{id \otimes \Delta_*} \\
H_{G}^{\mu'+\mu}\left(E^{p'+p,q'+q}_1\right) \otimes H_G^{-\rho}\left(E^1_{s,t}(X \times X)\right) \ar[d]_K & &  H_G^{\mu'}\left(E^{p',q'}_1\right) \otimes H_{G}^{-\rho + \mu} \left(E_{s-p,t-q}^1(X \times X)\right) \ar[d]^K \\
H_{G \times G}^{- \rho + \mu'+\mu}\left(E^{p'+p,q'+q}_1 \otimes E^1_{s,t}(X \times X)\right) \ar[d]_T & &  H_{G \times G}^{-\rho + \mu + \mu'}\left(E^{p',q'}_1 \otimes E_{s-p,t-q}^1(X \times X)\right) \ar[d]^T \\
H_{G}^{-\rho + \mu'+\mu}\left(E^{p'+p,q'+q}_1 \otimes E^1_{s,t}(X \times X)\right) \ar[rd]_{h \circ (id \otimes u^{-1})} & &  H_{G}^{-\rho + \mu + \mu'}\left(E^{p',q'}_1 \otimes E_{s-p,t-q}^1(X \times X)\right) \ar[ld]^{h \circ (id \otimes u^{-1})} \\
& H_{G}^{-(\rho - (\mu'+\mu))}\left(E_{s-(p'+p),t-(q'+q)}^1\right) & 
}}
\tag{4.10}
\end{equation}

Similarly to what we did to prove the properties of the cup product on ${}^*\!E_2$, we show the commutativity of previous diagram (\ref{equivcapformcup}) by filling it with commutative diagrams. We do in the following way. The left and right parts of (\ref{equivcapformcup}) become respectively :

\begin{landscape} 
{\tiny 
\begin{equation*} \xymatrix{
&&& H_G^{\mu'}\left(E_1^{p',q'}\right) \otimes H_G^{\mu}\left(E_1^{p,q}\right) \otimes H_G^{-\rho}\left(E^1_{s,t}\right) \ar[dl]_{K \otimes id} \\
& H_{G \times G}^{\mu'+\mu}\left(E^{p'+p,q'+q}_1(X \times X)\right) \otimes H_G^{-\rho}\left(E^1_{s,t}\right) \ar[d]_{T \otimes id} & H_{G \times G}^{\mu'+\mu}\left(E^1_{p',q'} \otimes E^1_{p,q}\right) \otimes H_G^{-\rho}\left(E^1_{s,t}\right) \ar[l]_{\raise0.3cm\hbox{$(u^{\vee})^{-1} \circ w \otimes id$}} \ar[r]^K \ar[d]^{T \otimes id} & H_{G \times G \times G}^{-\rho + \mu' + \mu} \left(E_1^{p',q'} \otimes E_1^{p,q} \otimes E^1_{s,t}\right) \ar[dd]^{T_0} \ar[ddl]^{T_1} \\
H_G^{\mu' + \mu}\left(E^{p'+p,q'+q}_1\right) \otimes H_G^{-\rho}\left(E^1_{s,t}\right) \ar[d]_{id \otimes \Delta_*} & H_G^{\mu' + \mu}\left(E^{p'+p,q'+q}_1(X \times X)\right) \otimes H_G^{-\rho}\left(E^1_{s,t}\right) \ar[l]_{\raise0.3cm\hbox{$\Delta^* \otimes id$}} \ar[d]_K \ar[ld]_{\Delta^* \otimes \Delta_*} & H_{G}^{\mu'+\mu}\left(E_1^{p',q'} \otimes E_1^{p,q}\right) \otimes H_G^{-\rho}\left(E^1_{s,t}\right) \ar[l]_{\raise0.3cm\hbox{$(u^{\vee})^{-1} \circ w \otimes id$}} \ar[d]^K & \\
H_{G}^{\mu'+\mu}\left(E^{p'+p,q'+q}_1\right) \otimes H_G^{-\rho}\left(E^1_{s,t}(X \times X) \right)  \ar[d]_K & H_{G \times G}^{-\rho + \mu' + \mu}\left(E^{p'+p,q'+q}_1(X \times X) \otimes E^1_{s,t}\right) \ar[ld]_{\Delta^* \otimes \Delta_*} \ar[rrd]^T \ar[d]_{\Delta^* \otimes id} & H_{G \times G}^{-\rho + \mu'+\mu }\left(E_1^{p',q'} \otimes E_1^{p,q} \otimes E^1_{s,t}\right) \ar[l]_{\raise0.3cm\hbox{$(u^{\vee})^{-1} \circ w \otimes id$}} \ar[r]^T & H_{G}^{-\rho + \mu'+\mu }\left(E_1^{p',q'} \otimes E_1^{p,q} \otimes E^1_{s,t}\right) \ar[d]^{(u^{\vee})^{-1} \circ w \otimes id} \\
H_{G \times G}^{-\rho+\mu'+\mu}\left( E^{p'+p,q'+q}_1 \otimes E^1_{s,t}(X \times X) \right) \ar[d]_T & H_{G \times G}^{-\rho+\mu'+\mu}\left( E^{p'+p,q'+q}_1 \otimes E^1_{s,t} \right) \ar[l]_{id \otimes \Delta_*} \ar[r]^T & H_{G}^{-\rho+\mu'+\mu}\left( E^{p'+p,q'+q}_1 \otimes E^1_{s,t} \right) \ar[lld]_{id \otimes \Delta_*} & H_{G}^{-\rho+\mu'+\mu}\left( E^{p'+p,q'+q}_1(X \times X) \otimes E^1_{s,t} \right) \ar[l]_{\Delta^* \otimes id} \\
H_{G}^{-\rho+\mu'+\mu}\left( E^{p'+p,q'+q}_1 \otimes E^1_{s,t}(X \times X) \right) \ar[rrr]^{h \circ (id \otimes u^{-1})} & & & H_{G}^{-(\rho - (\mu'+\mu))}\left(E_{s-(p'+p),t-(q'+q)}^1\right)
}
\end{equation*}
}
\end{landscape}

\begin{landscape} 
{\tiny 
\begin{equation*} \xymatrix{
H_G^{\mu'}\left(E_1^{p',q'}\right) \otimes H_G^{\mu}\left(E_1^{p,q}\right) \otimes H_G^{-\rho}\left(E^1_{s,t}\right) \ar[rr]^{id \otimes id \otimes \Delta_*} \ar[rd]^{id \otimes K} & & H_G^{\mu'}\left(E_1^{p',q'}\right) \otimes H_G^{\mu}\left(E_1^{p,q}\right) \otimes H_G^{-\rho}\left(E^1_{s,t}(X \times X) \right) \ar[d]^{id \otimes K} &  \\
H_{G \times G \times G}^{-\rho + \mu' + \mu} \left(E_1^{p',q'} \otimes E_1^{p,q} \otimes E^1_{s,t}\right) \ar[dd]_{T_0} \ar[rdd]_{T_2} & H_G^{\mu'}\left(E_1^{p',q'} \right) \otimes H_{G \times G}^{-\rho + \mu}\left(E_1^{p,q} \otimes E^1_{s,t}\right) \ar[l]_K \ar[r]^{\raise0.3cm\hbox{$id \otimes (id \otimes \Delta_*)$}} \ar[d]_{id \otimes T} & H_G^{\mu'}\left(E_1^{p',q'} \right) \otimes H_{G \times G}^{-\rho + \mu}\left(E_1^{p,q} \otimes E^1_{s,t}(X \times X)\right) \ar[d]^{id \otimes T} & \\
& H_G^{\mu'}\left(E_1^{p',q'} \right) \otimes H_{G}^{-\rho + \mu}\left(E_1^{p,q} \otimes E^1_{s,t}\right) \ar[r]^{\raise0.3cm\hbox{$id \otimes (id \otimes \Delta_*)$}} \ar[d]_K & 
H_G^{\mu'}\left(E_1^{p',q'} \right) \otimes H_{G}^{-\rho + \mu}\left(E_1^{p,q} \otimes E^1_{s,t}(X \times X)\right) \ar[r]^{\raise0.3cm\hbox{$id \otimes h(id \otimes u^{-1})$}} \ar[d]^K & H^{\mu'}_G\left(E_1^{p',q'}\right) \otimes H_G^{-\rho + \mu} \left(E^1_{s-p,t-q}\right) \ar[dddd]^{id \otimes \Delta_*}  \ar[lddd]^K \\
H_{G}^{-\rho + \mu' + \mu} \left(E_1^{p',q'} \otimes E_1^{p,q} \otimes E^1_{s,t}\right) \ar[d]_{id \otimes id \otimes \Delta_*} & H_{G \times G}^{-\rho + \mu' + \mu} \left(E_1^{p',q'} \otimes E_1^{p,q} \otimes E^1_{s,t}\right) \ar[r]^{\raise0.3cm\hbox{$id \otimes id \otimes \Delta_*$}} \ar[l]_T & H_{G \times G}^{-\rho + \mu' + \mu} \left(E_1^{p',q'} \otimes E_1^{p,q} \otimes E^1_{s,t}(X \times X)\right) \ar[dd]^{id \otimes h(id \otimes u^{-1})} \ar[lld]_{T} &  \\
H_{G}^{-\rho + \mu' + \mu} \left(E_1^{p',q'} \otimes E_1^{p,q} \otimes E^1_{s,t}(X \times X)\right) \ar[d]_{id \otimes h(id \otimes u^{-1})} & &  &  
\\
H_G^{-\rho + \mu' + \mu}\left(E^{p',q'}_1 \otimes E^1_{s-p,t-q}\right) \ar[d]_{id \otimes \Delta_*} & & H_{G \times G}^{-\rho + \mu' + \mu}\left(E^{p',q'}_1 \otimes E^1_{s-p,t-q}\right) \ar[d]^{id \otimes \Delta_*} \ar[ll]_{T} & \\
H_{G}^{-\rho + \mu' + \mu}\left(E^{p',q'}_1 \otimes E^1_{s-p,t-q}(X \times X) \right) \ar[d]_{h(id \otimes u^{-1})} & &H_{G \times G}^{-\rho + \mu' + \mu}\left(E^{p',q'}_1 \otimes E^1_{s-p,t-q}(X \times X) \right) \ar[ll]_{T}   & H^{\mu'}_G\left(E_1^{p',q'}\right) \otimes H_G^{-\rho + \mu} \left(E^1_{s-p,t-q}(X \times X)\right) \ar[l]^{K} \\
H_{G}^{-\rho + \mu' + \mu}\left(E^1_{s-p-p',t-q-q'}\right) &&&
} 
\end{equation*}
}
\end{landscape}

Finally, since the morphisms in the above diagrams are defined on the filtered chain level, the formula (\ref{equivcapformcup}) is also true on the induced spectral sequences ${}^*\!E_r$ and ${}^*\!E^r$ for $r \geq 2$, and ${}^G\!E_r$ and ${}^G\!E^r$ for $r \geq 1$. 

\end{proof}

\begin{theo} \label{prop2equivcap} Let $Y$ be a real algebraic $G$-variety and $f : X \rightarrow Y$ an equivariant morphism. Let $r \geq 1$ and let $p,q,s,t \in \mathbb{Z}$. If $\varphi \in {}^G\!E_r^{p,q}(Y)$ and $c \in {}^G\!E^r_{s,t}(X)$, then
\begin{equation} \label{equivformcapfunct} \varphi \frown f_*(c) = f_* (f^*(\varphi) \frown c). 
\tag{4.11}
\end{equation}
\end{theo}

\begin{proof} We first prove that formula (\ref{equivformcapfunct}) on the non-equivariant weight spectral sequences : suppose $\varphi \in E_r^{p,q}(Y)$ and $c \in E^r_{s,t}(X)$, then, for any $\psi \in \in E_r^{s-p,t-q}(Y)$, we have
\begin{eqnarray*} 
\psi\left(\varphi \frown f_*(c)\right) & = & \left(\psi \smile \varphi\right)\left(f_*(c)\right) \\
				       & = & f^*\left(\psi \smile \varphi\right)(c) \\
				       & = & \left(f^*(\psi) \smile f^*(\varphi)\right)(c) \\
				       & = & f^*(\psi)\left(f^*(\varphi) \frown c\right) \\
				       & = & \psi \left(f_*\left(f^*(\varphi) \frown c\right)\right).
\end{eqnarray*}
We then express formula (\ref{equivformcapfunct}) with the commutative diagram

\begin{equation*} \xymatrix{
& E^{p,q}_r(Y) \otimes E^r_{s,t}(X) \ar[ld]_{id \otimes f_*} \ar[rd]^{f^* \otimes id} & \\
E^{p,q}_r(Y) \otimes E^r_{s,t}(Y) \ar[d]_{id \otimes \Delta_*} & & E^{p,q}_r(X) \otimes E^r_{s,t}(X) \ar[d]^{id \otimes \Delta_*} \\
E^{p,q}_r(Y) \otimes E^r_{s,t}(Y \times Y)  \ar[rdd]_{h(id \otimes u^{-1})} & & E^{p,q}_r(X) \otimes E^r_{s,t}(X \times X) \ar[d]^{h(id \otimes u^{-1})}\\
& & E^r_{s-p,t-q}(X) \ar[ld]^{f_*}  \\
& E^r_{s-p,t-q}(Y) &
}
\end{equation*}

Finally, we prove the formula (\ref{equivformcapfunct}) on ${}^*\!E_2$ and ${}^*\!E^2$, showing that the following diagram is commutative :
\begin{equation*} \hspace{-1.5cm} \xymatrix{
& H_G^{\mu}\left(E^{p,q}_1(Y)\right) \otimes H_G^{-\rho}\left(E^1_{s,t}(X)\right) \ar[ld]_{id \otimes f_*} \ar[rd]^{f^* \otimes id} & \\
H_G^{\mu}\left(E^{p,q}_1(Y)\right) \otimes H_G^{-\rho}\left(E^1_{s,t}(Y)\right) \ar[d]_{id \otimes \Delta_*} & & H_G^{\mu}\left(E^{p,q}_1(X)\right) \otimes H_G^{-\rho}\left(E^1_{s,t}(X)\right) \ar[d]^{id \otimes \Delta_*} \\
H_G^{\mu}\left(E^{p,q}_1(Y)\right) \otimes H_G^{-\rho}\left(E^1_{s,t}(Y \times Y)\right) \ar[d]_{K}  & & H_G^{\mu}\left(E^{p,q}_1(X)\right) \otimes H_G^{-\rho}\left(E^1_{s,t}(X \times X)\right) \ar[d]^K \\
H^{-\rho + \mu}_{G \times G}\left(E^{p,q}_1(Y) \otimes E^1_{s,t}(Y \times Y)\right) \ar[d]_{T} & & H^{-\rho + \mu}_{G \times G}\left(E^{p,q}_1(X) \otimes E^1_{s,t}(X \times X)\right) \ar[d]^{T} \\
H^{-\rho + \mu}_{G}\left(E^{p,q}_1(Y) \otimes E^1_{s,t}(Y \times Y)\right) \ar[rdd]_{h(id \otimes u^{-1})} & & H^{-\rho + \mu}_{G}\left(E^{p,q}_1(X) \otimes E^1_{s,t}(X \times X)\right) \ar[d]^{h(id \otimes u^{-1})}  \\
& & H^{-\rho + \mu}_G\left(E^1_{s-p,t-q}(X)\right) \ar[ld]^{f_*}  \\
& H^{-\rho + \mu}_G\left(E^1_{s-p,t-q}(Y)\right) &
}
\end{equation*}
To this end, we fill it by commutative diagrams as follows. Since the involved morphisms are defined on the filtered chain level, the formula (\ref{equivformcapfunct}) is induced on ${}^*\!E_r$ and ${}^*\!E^r$ for $r \geq 2$, and on ${}^G\!E_r$ and ${}^G\!E^r$ for $r \geq 1$.
\begin{landscape}
{\tiny
\begin{equation*} \xymatrix{
& & H_G^{\mu}\left(E^{p,q}_1(Y)\right) \otimes H_G^{-\rho}\left(E^1_{s,t}(X)\right) \ar[lld]_{id \otimes f_*} \ar[rrd]^{f^* \otimes id} \ar[d]_K & & \\
H_G^{\mu}\left(E^{p,q}_1(Y)\right) \otimes H_G^{-\rho}\left(E^1_{s,t}(Y)\right) \ar[d]_{id \otimes \Delta_*} & & H^{-\rho + \mu}_{G \times G}\left(E^{p,q}_1(Y) \otimes E^1_{s,t}(X) \right) \ar[d]_T \ar[ldd]_{id \otimes f_*}  \ar[rdd]^{f^* \otimes id}  & & H_G^{\mu}\left(E^{p,q}_1(X)\right) \otimes H_G^{-\rho}\left(E^1_{s,t}(X)\right) \ar[d]^{id \otimes \Delta_*} \\
H_G^{\mu}\left(E^{p,q}_1(Y)\right) \otimes H_G^{-\rho}\left(E^1_{s,t}(Y \times Y)\right) \ar[d]_{K}  & & H^{-\rho + \mu}_{G}\left(E^{p,q}_1(Y) \otimes E^1_{s,t}(X) \right) \ar[ldd]_{id \otimes f_*}  \ar[rdd]^{f^* \otimes id} & & H_G^{\mu}\left(E^{p,q}_1(X)\right) \otimes H_G^{-\rho}\left(E^1_{s,t}(X \times X)\right) \ar[d]^K \\
H^{-\rho + \mu}_{G \times G}\left(E^{p,q}_1(Y) \otimes E^1_{s,t}(Y \times Y)\right) \ar[d]_{T} & H^{-\rho + \mu}_{G \times G}\left(E^{p,q}_1(Y) \otimes E^1_{s,t}(Y)\right)  \ar[l]_{id \otimes \Delta_*} \ar[d]_T & & H^{-\rho + \mu}_{G \times G}\left(E^{p,q}_1(X) \otimes E^1_{s,t}(X)\right)  \ar[r]^{id \otimes \Delta_*} \ar[d]^T & H^{-\rho + \mu}_{G \times G}\left(E^{p,q}_1(X) \otimes E^1_{s,t}(X \times X)\right) \ar[d]^{T} \\
H^{-\rho + \mu}_{G}\left(E^{p,q}_1(Y) \otimes E^1_{s,t}(Y \times Y)\right) \ar[rrdd]_{h(id \otimes u^{-1})} & H^{-\rho + \mu}_{G}\left(E^{p,q}_1(Y) \otimes E^1_{s,t}(Y)\right) \ar[l]_{id \otimes \Delta_*} & & H^{-\rho + \mu}_{G}\left(E^{p,q}_1(X) \otimes E^1_{s,t}(X)\right) \ar[r]^{id \otimes \Delta_*} & H^{-\rho + \mu}_{G}\left(E^{p,q}_1(X) \otimes E^1_{s,t}(X \times X)\right) \ar[d]^{h(id \otimes u^{-1})}  \\
& & & & H^{-\rho + \mu}_G\left(E^1_{s-p,t-q}(X)\right) \ar[lld]^{f_*}  \\
& & H^{-\rho + \mu}_G\left(E^1_{s-p,t-q}(Y)\right) & &
}
\end{equation*} 
}
\end{landscape}

\end{proof}

\begin{rem} Let $X$ be a real algebraic $G$-variety of dimension $d$. The semialgebraic chain $[X]$ belongs to $\mathcal{G}_{-d} C_d(X)$. Notice that the complex $\mathcal{G}_{-d} C_*(X)$ is concentrated in degree $d$ and therefore, we have
\begin{itemize}
	\item $\Lambda_{-d} C_k^G(X) = Hom_G(F_{-d+k}, \mathcal{G}_{-d} C_d(X))$
	\item ${}^{d}\!E^{\alpha, \beta}_2 = {}^{d}\!E^{\alpha, \beta}_{\infty} = \begin{cases} H^{-\alpha}(G, \mathcal{G}_{-d} C_d(X)) & \mbox{ if $\beta = 0$,} \\ 0 & \mbox{ otherwise,} \end{cases}$
	\item ${}^G\!\widetilde{E}_2^{p,d} = H^{-p}(G, \mathcal{G}_{-d} C_d(X))$.
\end{itemize}

In particular, the semialgebraic chain $[X]$ can be considered as a class of ${}^{d}\!E^{0, 0}_2 = {}^G\!\widetilde{E}_2^{0,d} = {}^G\!E_1^{-d,2d} = {}^G\!E_{\infty}^{-d,2d} = (\mathcal{G}_{-d} C_d(X))^G$ and we can consider the equivariant cap product with $[X]$ 
$${}^p\!D : {}^p\!E^{\alpha, \beta}_r \rightarrow {}^{d-p}\!E^r_{-\alpha,-\beta}$$
on ${}^*\!E_r$, for $r \geq 2$, 
$${}^G\!D : {}^G\!E_r^{p,q} \rightarrow {}^G\!E^r_{-d-p,2 d-q}$$
on ${}^G\!E_r$, for $r \geq 1$, and
$$D_G : \Omega^p H^k(X ; G) \rightarrow \Omega_{-p-n} H_{n-k}(X ; G)$$
on $H^*(X ; G)$.
\\

When $X$ is compact nonsingular, 
$$\widetilde{E}_2^{p,q} = \widetilde{E}_{\infty}^{p,q} = \begin{cases} H^q(X) & \mbox{ if $q = 0$,} \\ 0 & \mbox{ otherwise,} \end{cases}$$
and the spectral sequences ${}^*_I\!E$ degenerate at page two, the reindexed equivariant weight spectral sequences of $X$ being, from page $r =2$, the Hochschild-Serre spectral sequences associated to $X$ and $G$ (remark \ref{remaddcohom} and Proposition 3.23 of \cite{Pri-EWF}). Furthermore, the Poincar\'e duality isomorphism
$$D : H^k(X) \rightarrow H_{n-k}(X),$$
which is induced by morphisms defined on the (filtered) chain level (see section 5.6 of \cite{LP}), induces an isomorphism of Hochschild-Serre spectral sequences from page two :
$$D_G: {}^G\!\widetilde{E}_2^{p,q} = H^{p}(G , H^q(X)) \rightarrow H^p(G, H_{d-q}(X)) = {}^G\!\widetilde{E}^2_{-p,d-q}.$$

This isomorphism of spectral sequences then induces an equivariant Poincar\'e duality isomorphism on the equivariant cohomology and homology of $X$
$$D_G : \Omega^p H^k(X ; G) \rightarrow \Omega_{-p-n} H_{n-k}(X ; G)$$
(here, the equivariant weight filtration $\Omega$ coincides with the filtration induced by the Hochschild-Serre spectral sequence), which coincides with the equivariant Poincar\'e duality of \cite{VH} Theorem 4.2 (see also Remark 5.3).
\end{rem}

 \vspace{0.5cm}
Fabien PRIZIAC
\\
Institut de Math\'ematiques de Marseille
\\
(UMR 7373 du CNRS)
\\
Aix-Marseille Universit\'e
\\
39, rue Fr\'ed\'eric Joliot Curie
\\
13453 Marseille Cedex 13, France
\\
fabien.priziac@univ-amu.fr
\\

Aix Marseille Univ, CNRS, Centrale Marseille, I2M, Marseille, France


\begin{thebibliography}{16}

\bibitem{Bro}{K.S. Brown, \emph{Cohomology of groups}, Graduate texts in Mathematics, {\bf 87}, Springer-Verlag, 1982.}  

\bibitem{CTVZ}{J.F. Carlson, L. Townsley, L. Valero-Elizondo, M. Zhang, \emph{Cohomology Rings of Finite Groups with an Appendix: Calculations of Cohomology Rings of Groups of Order Dividing 64}, Algebra and Applications, Kluwer Academic Publishers, Dordrecht, 2003.}


\bibitem{Del}{P. Deligne, \emph{Poids dans la cohomologie des vari\'et\'es alg\'ebriques}, Proc. Int. Cong. Math. Vancouver (1974), 79-85.}


\bibitem{GF}{G. Fichou, \emph{Equivariant virtual Betti numbers}, Ann. de l'Inst. Fourier, {\bf 58}, no. 1 (2008), 1-27.}

\bibitem{GNAPP}{F. Guillén, V. Navarro Aznar, P. Pascual, F. Puerta, {\it Hyperrésolutions cubiques et descente cohomologique}, Lect. Notes in Math., {\bf 1335}, Springer-Verlag, Berlin 1988.}

\bibitem{GNA}{F. Guill\'en, V. Navarro Aznar, \emph{Un crit\`ere d'extension des foncteurs d\'efinis sur les sch\'emas lisses}, IHES Publ. Math. {\bf 95} \rm (2002), 1-83.}


\bibitem{Kur}{K. Kurdyka, \emph{Ensembles semi-alg\'ebriques sym\'etriques par arcs}, Math. Ann. {\bf 281} (1988), 445-462.} 

\bibitem{KP}{K. Kurdyka, A. Parusi\'nski, \emph{Arc-symmetric sets and arc-analytic mappings}, Panoramas et Synth\`eses {\bf 24}, Soc. Math. France (2007), 33-67.}

\bibitem{LP}{T. Limoges, F. Priziac, \emph{Cohomology and products of real weight filtrations}, Annales de l'Institut Fourier, 2015, {\bf 65} (5), pp.2235-2271, doi:10.5802/aif.2987}

\bibitem{MacLane}{S. Mac Lane, Homology, Springer-Verlag, Berlin, 1963.}


\bibitem{MCP-VB}{C. McCrory, A. Parusi\'nski, \emph{Virtual Betti numbers of real algebraic varieties}, C. R. Math. Acad. Sci. Paris {\bf 336} (2003), no. 9, 763-768.}

\bibitem{MCP}{C. McCrory, A. Parusi\'nski, \emph{The weight filtration for real algebraic varieties}, Topology of stratified spaces, 121-160, Math. Sci. Res. Inst. Publ. {\bf 58}, Cambridge Univ. Press, Cambridge, 2011.}

\bibitem{Pri-CA}{F. Priziac, \emph{Complexe de poids des vari\'et\'es alg\'ebriques r\'eelles avec action}, Mathematische Zeitschrift {\bf 277}, issue 1 (2014), 63-80, doi:10.1007/s00209-013-1244-8}

\bibitem{Pri-EWF}{F. Priziac, \emph{Equivariant weight filtration for real algebraic varieties with action}, Journal of the Mathematical Society of Japan, 2016, {\bf 68} (4), pp 1789-1818, doi :10.2969/jmsj/06841789}

\bibitem{Tot}{B. Totaro, \emph{Topology of singular algebraic varieties}, Proc. Int. Cong. Math. Beijing (2002), 533-541.} 

\bibitem{VH}{J. van Hamel, \emph{Algebraic cycles and topology of real algebraic varieties}, CWI Tract 129, Stichting Mathematisch Centrum, Centrum voor Wiskunde en Informatica, Amsterdam, 1997.}

\end{thebibliography}
\end{document}